\documentclass{article}
\usepackage[utf8]{inputenc}
\usepackage[english]{babel}

\usepackage[T1]{fontenc}

\usepackage{style}
\usepackage{lmodern}

\title{On birational automorphisms of double EPW-cubes}
\date{\today}
\author{Simone Billi\footnote{Member of the INdAM group GNSAGA, and partially supported by the Curiosity Driven 2021 Project \textit{Varieties with trivial or
negative canonical bundle and the birational geometry of moduli spaces of curves: a
constructive approach} - Programma nazionale per la Ricerca (PNR) DM 737/2021} \and Stevell Muller\footnote{Supported by the Deutsche Forschungsgemeinschaft (DFG, German Research Foundation) – Project-ID 286237555 –
TRR 195 (Gefördert durch die Deutsche Forschungsgemeinschaft (DFG) – Projektnummer 286237555 – TRR 195)} \and Tomasz Wawak\footnote{Supported by  Narodowe Centrum Nauki 2018/30/E/ST1/0053 grant,  the Narodowe Centrum Nauki grant UMO-2021/43/D/ST1/02290, and the Priority Research Area SciMat under the program Excellence
Initiative – Research University at the Jagiellonian University in Kraków}}

\AtEndDocument{\bigskip{\footnotesize%
  \textsc{Simone Billi, Dipartimento di Matematica, Universit\`a degli Studi di Genova, Via Dodecaneso 35, 16146 Genova, Italy} \par  
  \textit{E-mail address}: \texttt{name(dot)lastname [at] edu.unige.it} \par
  \addvspace{\medskipamount}
  \textsc{Stevell Muller, Fakult\"at f\"ur Mathematik und Informatik, Universit\"at des Saarlandes, Campus E2.4, 66123 Saarbr\"ucken, Germany} \par
  \textit{E-mail address}: \texttt{lastname [at] math.uni-sb.de} \par
  \addvspace{\medskipamount}
  \textsc{Tomasz Wawak, Wydział Matematyki i Informatyki, Uniwersytet Jagielloński, ul. Prof. S. Łojasiewicza 6, 30-348 Kraków, Poland}\par
  \textit{E-mail address}: \texttt{name(dot)lastname [at] doctoral.uj.edu.pl}
}}

\begin{document}
\fontsize{10pt}{14pt}
\maketitle

\begin{abstract}
We give a classification of finite groups of symplectic birational automorphisms on a manifold of \(\KTcube\)--type with stable cohomological action. We describe the group of  polarized automorphisms of a smooth double EPW-cube. Using this description, we exhibit examples of projective hyperk\"ahler manifolds of \(\KTcube\)--type of maximal Picard rank with a symplectic action of a large group.
\end{abstract}
\section*{Introduction}

The classification of finite symplectic birational actions on the known deformation types of hyperk\"ahler manifolds is still an open problem. 
Nonetheless, the works \cite{nikulin-symp,kondo-symp,mukai-symp,hashimoto-symp} about automorphisms of K3 surfaces paved a way of resolving this question using lattice-theoretic arguments. 
In particular, there exists a complete lattice classification of finite symplectic actions on K3 surfaces.\smallskip

Beyond the case of prime order cyclic actions, the actual classification of finite groups of symplectic birational automorphisms for higher dimensional hyperk\"ahler manifolds is more intricate than in the case of K3 surfaces.
There are three main reasons: there may be infinitely many deformation types per dimension, the associated quadratic forms are not unimodular, and the associated monodromy groups are not always maximal. In \cite[Theorem]{giosympaut},  Mongardi shows that finite groups of symplectic automorphisms on hyperk\"ahler manifolds of deformation type $\textnormal{K3}^{[n]}$ can be recovered from finite groups of isometries of the Leech lattice. 
This argument allowed H\"ohn and Mason to complete the cohomological description of finite groups of symplectic automorphisms for the deformation type $\textnormal{K3}^{[2]}$ \cite{Hohn2014FiniteGO}. 
Similar lattice-theoretic methods were used in the classification of finite symplectic actions on cubic fourfolds \cite{laza-zheng} and also, to some extent, in the recent classification of finite groups of symplectic birational transformations for the type OG10 \cite{marqinvo,marquand-muller}.\smallskip

We classify stable, i.e. with trivial action on the discriminant group\footnote{Here stable is not related to any stability condition, but trivial action on the discriminant group is stable under primitive embeddings of lattices, see \cite{GHS}}, finite subgroups of symplectic birational transformations on hyperk\"ahler manifolds of $\KTcube$--type. If \(X\) is hyperkähler, then we denote by \(\Bir^s(X)\) its group of symplectic birational transformations.
\begin{thm}[{\autoref{class_symplectic_groups}}]
    Let \(G\) be a finite group. Then there exists a hyperk\"ahler manifold \(X\) of \(\KTcube\)--type such that \(G\leq\Bir^s(X)\) is stable if and only if \(G\) is isomorphic to a subgroup of ones of the groups in \cite{database} (see also \autoref{tab: lovely table 2}).
\end{thm}
The proof is based on an application of the global Torelli theorem for hyperk\"ahler manifolds, which reduces the problem to classifying finite groups of isometries of a given even lattice. The result of Mongardi \cite[Theorem]{giosympaut} relates groups of birational transformations with certain finite groups acting on the Leech lattice, which have been classified by H\"ohn--Mason \cite{HMLeech}. The core of the proof is computer aided: we refer to \autoref{appendic lattice computation} for more details.\smallskip

Many of the groups of automorphisms described in \autoref{tab: lovely table 2} can be realized as groups of natural automorphisms on the Hilbert scheme of points on a K3 surface. An interesting problem is to produce explicit examples of manifolds of \(\KTcube\)--type with a symplectic action of a large group, especially of groups with maximal associated coinvariant sublattice. We realize geometrically some of the groups in the classification as groups of symplectic automorphisms (or birational automorphisms) of (desingularized) double EPW-cubes. \smallskip

Let \(V_6\) be a \(6\)-dimensional complex vector space and consider \(\LG\), the Grassmannian of Lagrangian subspaces of \(\bigwedge^3 V_6\). 
In \cite{iliev2019epw}, for a given Lagrangian subspace \([A]\subset\bigwedge^3 V_6\), the authors define an associated variety called EPW-cube\footnote{The fact that its double cover is, in general, a manifold of \(\KTcube\)--type is where the "cube" in the EPW-cube name comes from, as opposed to EPW-sextics which are named based on degree of the variety.} \(\cub{2}\subset\Gr\) as the degeneracy locus of a certain map of vector bundles. They show that if \([A]\not\in(\Sigma\cup\Gamma)\), where \(\Sigma,\Gamma\subset \LG\) are two distinguished divisors which share no common components, see \autoref{distinguished_divisors},
then there exists a natural double cover \(\dcub{2}\rightarrow\cub{2}\) which is a hyperkähler manifold of \(\KTcube\)--type called a double EPW-cube.
According to \cite{debarre2020double}, there is a natural double cover of an EPW-cube even in the case \([A]\in\Gamma\setminus\Sigma\), but it is singular in a finite number of points. From the recent work of Rizzo \cite{rizzo}, we know that for \([A]\in\Gamma\setminus\Sigma\) there is a projective resolution \(p\colon X \rightarrow \dcub{2}\) where \(X\) is hyperkähler of \(\KTcube\)--type and \(p\) is given by \(\mathbb{P}^3\)--contractions. \smallskip

We describe the group of polarized automorphisms of a double EPW-cube associated to a Lagrangian space \([A]\not\in(\Gamma\cup\Sigma)\), see \autoref{automorphisms_double_cover}, in analogy to the case of double EPW-sextics in \cite[Appendices A \& B]{debarre2022gushel}. Moreover, we show that for \([A]\not\in\Sigma\), the stabilizer of $A$ under the action of $\PGL(V_6)$ on \(\LG\) induces symplectic birational automorphisms of the double EPW-cube or one of its desingularizations in the singular case.
As an application, we find explicit examples of double EPW-cubes of maximal Picard rank and with a symplectic action of a given finite group \(G\), induced by a linear representation of \(G\) on \(V_6\). 
The variety \(\dcub{2}\) has a natural polarization \(H\) which is the pullback of the polarization of \(\cub{2}\): we denote by \(\Aut_H(\dcub{2})\) the group of automorphisms preserving $H$.

\begin{thm}[{\autoref{symmetric examples}}]\label{main_thm_3}
For any finite group \(G\) in \autoref{group_actions}, there exists a faithful projective representation \(G\rightarrow \PGL(V_6)\) and an associated \(G\)-invariant Lagrangian space \([A]\in\LG\setminus\Sigma\) so that
\begin{enumerate}
    \item  if \(G=\mathcal{A}_7, M_{10},L_2(11)\) then \([A]\not\in\Gamma\), and therefore the associated double EPW-cube \(X\defeq \dcub{2}\) is smooth. Moreover
    \[\Aut_H(X)\cong\left\{\begin{array}{cl} G\times \mathbb{Z}/4\mathbb{Z}& \text{ if $G=L_2(11)$}\\ G\times \mathbb{Z}/2\mathbb{Z}&\text{ if $G = \mathcal{A}_7,\, M_{10}$}\end{array}\right.,\]
    where \(G\) corresponds to the subgroup of symplectic automorphisms \(\Aut_H^s(X)\).
    \item  if \(G=L_3(4)\), then \([A]\in\Gamma\), and the associated double EPW-cube \(\dcub{2}\) is singular. In this case we let \(X\) be a projective symplectic resolution, and there is a group embedding \(G\hookrightarrow \Bir_H^s(X)\).
\end{enumerate}
    In any case, \(X\) is a hyperkähler manifold of \(\KTcube\)--type, the transcendental lattice \(T(X)\) is given in \autoref{group_actions}, and the entry \textnormal{Regular} is \textnormal{True} if the action of \(G\) on \(X\) is regular, \textnormal{False} otherwise.\vspace*{-3ex}
\end{thm}
{
\renewcommand
\arraystretch{2}
\begin{table}[H]
\caption{Groups acting symplectically on double EPW-cubes of maximal Picard rank}\label{group_actions}\vspace*{0.3cm}
\centering
    \small{\begin{tabular}{|c||c|c|c|c|c}
\hline
 $G$& $L_3(4)$&$\mathcal{A}_7$ & $M_{10}$ & $L_2(11)$ \\
 \hline
 Regular&False&True&True&True\\
 \hline
 $\T(X)$& $\begin{pmatrix}10&4\\4&10\end{pmatrix}$ &$\begin{pmatrix}6&0\\0&70\end{pmatrix}$ & $\begin{pmatrix}4&0\\0&30\end{pmatrix}$ & $\begin{pmatrix}22&0\\0&22\end{pmatrix}$\\
 \hline
\end{tabular}}

\end{table}}
These examples are constructed similarly to the very symmetric double EPW-sextics in \cite{debarre2022gushel,billi2022double}, 
the description of transcendental lattices uses similar techniques to the case of K3 surfaces in \cite{brandhorst-hashimoto} and to the case of hyperk\"ahler manifolds of \(\KTsquare\)--type in \cite{wawak2022very}. These groups of automorphisms cannot be realized as natural groups of automorphisms on an Hilbert scheme of points on a K3 surface, and are certainly not the only groups that can be realized on double EPW-cubes.\smallskip

The structure of the paper is the following. In \autoref{sec: prelim} we recall known facts about lattices, hyperkähler manifolds and double EPW-cubes. In \autoref{appendic lattice computation} we give a partial classification of groups of symplectic birational automorphisms of hyperk\"ahler manifolds of \(\KTcube\)--type. In \autoref{sec describe aut} we describe the group of polarized automorphisms of a smooth double EPW-cube associated to general enough Lagrangian spaces. In \autoref{sec general bir} we compute the full group of birational automorphisms of the general (desingularized) double EPW-cube associated to a general $[A]\notin\Sigma$. Finally, in \autoref{sec very sym}, we exhibit examples of symmetric (desingularized) double EPW-cubes of maximal Picard rank.
In \autoref{sec appendix}, we display in a table the results from the classification procedure described in \autoref{appendic lattice computation}.

\section*{Acknowledgements}
We would like to thank Francesco Denisi and Marco Rampazzo for discussing part of the project at its early stage. We would also like to thank Simon Brandhorst, Grzegorz Kapustka, Michał Kapustka and Giovanni Mongardi for their guidance and fruitful discussions, as well as Francesca Rizzo for useful comments on an early version of the article. Finally, we would like to thank referees for their comments and suggestions which helped us improve the quality of the paper.

\section{Preliminaries and notations}\label{sec: prelim}
In this section we collect preliminary notions and notations about lattice theory, hyperkähler manifolds, double EPW-sextics and double EPW-cubes.
\subsection{Lattice theory}
We define a \textit{lattice} as a finitely generated free $\mathbb{Z}$-module $L$ equipped with a nondegenerate bilinear form $b: L\times L\to \mathbb{Z}$. 
We often omit $b$ and use the notation $v^2 \defeq  b(v, v)$ for $v\in L$. A lattice \(L\) is \textit{even} if for any \(v\in L\) we have \(v^2\in2\mathbb{Z}\). The \textit{rank} of a lattice is the rank of the lattice as a \(\mathbb{Z}\)-module, and its \textit{signature} is the signature of the real form associated to it.

For a lattice $L$ and a vector $v\in L$, the \textit{divisibility} $\divisibility(v, L)$ is the positive generator of the ideal $b(v, L)\subseteq\mathbb{Z}$. 
Notice that $\divisibility(v, L)$ is the largest positive integer $\gamma$ such that $v/\gamma\in L^\vee\defeq\Hom_{\mathbb{Z}}(L,\mathbb{Z})$.
For $v\in L$, we call the pair $(v^2, \divisibility(v, L))$ the \textit{type} of $v$ (with respect to $L$).
    
Given a lattice \(L\), the finite abelian group $D_L\defeq  L^\vee/L$ is called the \textit{discriminant group} of $L$.
If the lattice $L$ is even, the group $D_L$ is equipped with a finite quadratic form taking values in \(\mathbb{Q}/2\mathbb{Z}\). There is a natural map between the groups of isometries
\[\pi:\bO(L)\to \bO(D_L)\]
whose kernel is denoted by $\bO^\#(L)$. Any isometry $f\in \bO^\#(L)$ or subgroup of isometries $G\leq \bO^\#(L)$ is called \emph{stable}.

Let $L$ be a lattice and let $g\in \bO(L)$. Since $\bO(L\otimes \mathbb{R})$ is generated by reflections, by the Cartan-Dieudonné theorem we can expressed $g$ over the reals as the composition of reflections along non-isotropic vectors $v_1, \ldots, v_r\in L\otimes \mathbb{R}$. We define the \emph{real spinor norm} of $g$, denoted $\sigma(g)$, as the class modulo $(\mathbb{R}^\times)^2$ of the product of the $-v_i^2/2\in\mathbb{R}^\times$. In this way, one obtains that $-\id_L = (-1)^p$ if $L$ has signature $(p, n)$. This defines a morphism
\[\sigma\colon \bO(L)\to \mathbb{R}^\times/(\mathbb{R}^\times)^2,\]
called \emph{spinor norm morphism}, whose kernel is denoted by $\bO^+(L)$. We also denote $\widetilde{\bO}(L) := \bO^+(L)\cap \bO^\#(L)$.

For \(G\leq \bO(L)\) a group of isometries, we define \(L^G\defeq \{v\in L\;\mid\; g(v)=v,\; \forall g\in G\}\) the \emph{invariant sublattice} and \(L_G\defeq (L^G)^\perp\) the \emph{coinvariant sublattice}.
 For a chain of subgroups $H\leq G\leq \bO(L)$, we define the \emph{saturation} of $H$ in $G$ to be the pointwise stabilizer of $L^H$ in $G$. We say $H$ is \emph{saturated} in $G$ if it coincides with its saturation. We often also say that $H$ is \emph{stably saturated} in $G$ if $H^\#$ is saturated in $G^\#$. If $G$ is understood from the context, or if $G = \bO(L)$, we simply say $H$ saturated or stably saturated.
 
\begin{lem}\label{lem stab sat}
    Let $L$ be an even lattice and let $G\leq \bO^\#(L)$. Then $G$ is stably saturated if and only if $\Img(G\to \bO(L_G)) = \bO^\#(L_G)$.
\end{lem}
\begin{proof}
    If $L$ is unimodular, the statement is clear: the groups \(D_{L^G}\) and \(D_{L_G}\) are anti-isometric and in this case the anti-isometry giving the primitive extension $L^G\oplus L_G\subset L$ is given by an isometry \(D_{L^G}\cong D_{L_G}(-1)\) \cite[Corollary 1.6.2]{nikulin}. Thus, according to \cite[Corollary 1.5.2]{nikulin}, any isometry of $L$ fixing $L^G$ is of the form $\id_{L^G}\oplus g$ where $g\in \bO^\#(L_G)$: in particular, $G$ is stably saturated if and only if $G$ restricts to $\bO^\#(L_G)$. 
    
    If $L$ is not unimodular, we embed it into an even unimodular lattice $M$, and we denote by $F := L^\perp_M$ its orthogonal complement so that $M$ is an overlattice of $F\oplus L^G\oplus L_G$. Since $G$ is stable, there exists a group $\widetilde{G} \leq \bO(M)$ which fixes $F$ and restrict to $G$ on $L$ \cite[Corollary 1.5.2]{nikulin}. Note that $L_G = M_{\widetilde{G}}$, and the saturation of $\widetilde{G}$ in $\bO(M)$, which fixes both $F$ and $L^G$ by definition of $\widetilde{G}$, restricts on $L$ to the saturation of $G$ in $\bO^\#(L)$. From the first part of the proof, we deduce that $G\xrightarrow{\cong}\widetilde{G}\to \bO^\#(L_G)$ is injective, and it is bijective if and only if $G$ is saturated in $\bO^\#(L)$.
\end{proof}

We denote by \(\bA_i,\bD_i, \bE_i\) the ADE root lattices, that are considered to be negative definite. We often describe a lattice $L$ by a Gram matrix associated to one of its bases. 
For instance, we denote by
\[\bU\defeq  \begin{pmatrix}
    0&1\\1&0
\end{pmatrix}\]
the integer hyperbolic plane lattice. 
Moreover, given a nonzero integer $n$ we denote by $[n]$ the rank one lattice $\mathbb{Z}v$ such that $v^2=n$.

\subsection{Hyperkähler manifolds}
A hyperkähler manifold is a simply-connected compact K\"ahler manifold \(X\) such that $\Homology^{2,0}(X) = \Homology^0(X, \bigwedge\nolimits^2\Omega_X) \cong \mathbb{C}\sigma_X$ where $\sigma_X$ is a nowhere degenerate holomorphic 2-form. 
For a hyperkähler manifold \(X\), the finitely generated free $\mathbb{Z}$-module $\Homology^2(X, \mathbb{Z})$ has a lattice structure given by the Beauville-Bogomolov-Fujiki form, 
of real signature \((3,b_2(X)-3)\). 
The N{\'e}ron-Severi group of $X$, which in this case is a lattice, is defined as
\begin{align*}
    \NS(X) \defeq  \Homology^{2}(X, \mathbb{Z}) \cap \Homology^{1,1}(X) \subset \Homology^2(X, \mathbb{C}).
\end{align*}
The transcendental lattice $\T(X)$ is its orthogonal complement with respect to the associated bilinear form. For the general theory of hyperkähler manifolds, we refer for example to \cite{debarre2022hyper}.

Given a hyperk\"ahler manifold $X$, the natural orthogonal representation
\[\rho_X\colon \Bir(X)\to \bO^+(\Homology^2(X, \mathbb{Z}))\]
has finite kernel and its image is completely described by the Strong Torelli theorem \cite[Theorem 1.3]{markman2011survey}. A birational automorphism $f\in \Bir(X)$ is called \emph{symplectic} if $f^\ast\sigma_X = \sigma_X$. We denote by $\Bir^s(X)$ and $\Aut^s(X)$ the groups of symplectic birational and symplectic regular automorphisms of $X$ respectively.

A birational automorphism $g\in\Bir(X)$ (resp. a subgroup $G\leq \Bir(X)$ of such) is called \emph{stable} if so is $\rho_X(f)$ (resp. $\rho_X(G)$). In a similar vain, we call a subgroup $G\leq\Bir(X)$ \emph{stably saturated} if so is $\rho_X(G)\leq \bO^+(\Homology^2(X, \mathbb{Z}))$.\bigskip

Finally, recall that if \(X\) is a manifold of K3\(^{[n]}\)--type, then we have an isometry \(\eta\colon\Homology^2(X,\mathbb{Z})\simeq \bL_{\textnormal{K3}^{[n]}}\) where \(\bL_{\textnormal{K3}^{[n]}}\defeq \bU^{\oplus 3}\oplus\bE_8^{\oplus2}\oplus[-2(n-1)]\).
The representation map \(\rho_X\) is injective if \(X\) is of K3\(^{[n]}\)--type \cite{beauville1983varietes} and hence
we can study the birational automorphisms of any $X$ of $\textnormal{K3}^{[n]}$--type by studying the isometries of \(\bL_{\textnormal{K3}^{[n]}}\). In this case we will simply write \(G\) instead of \(\rho_X(G)\).

\subsection{Symplectic automorphisms of manifolds of \texorpdfstring{\(\KTcube\)}{K3[3]}--type}
We recall the notions of prime exceptional divisors and wall divisors for manifolds of \(\KTcube\)--type, which play a crucial role in the lattice-theoretical classification of groups of symplectic birational transformations.

\begin{defin}[{{{\cite{hassett2010intersection,Markman_prime_exceptional}}}}]\label{wall_divisors}
The set of \emph{numerical prime exceptional divisors} for hyperk\"ahler manifolds of \(\KTcube\)--type is given by
\begin{align*}
\mathcal{W}_{\KTcube}^{pex}\defeq \left\{v\in \bL_{\KTcube}\mid (v^2,\divisibility(v,\bL_{\KTcube}))=(-2,1),(-4,2),(-4,4)\right\}.
\end{align*}
The set of \emph{numerical wall divisors} for hyperk\"ahler manifolds of \(\KTcube\)--type is given by
\begin{align*}
\mathcal{W}_{\KTcube}&\defeq \mathcal{W}_{\KTcube}^{pex}\cup\left\{v\in \bL_{\KTcube}\mid (v^2,\divisibility(v,\bL_{\KTcube}))=(-12,2),(-36,4)\right\}.\end{align*}
\end{defin}
We denote by \(\mathbb{L}\) the \textit{Leech lattice}, the unique negative definite unimodular lattice of rank \(24\) without vectors of square \(-2\) and consider the \textit{Conway group} \(Co_1 \defeq \bO(\mathbb{L})/\{\pm \id_\mathbb{L}\}\).

\begin{prop}[\cite{giosympaut}]\label{wall-div_criterion}
    Let $G$ be a finite group.
    \begin{enumerate}        \item 
    There exists a hyperk\"ahler manifold $X$ of $\KTcube$--type with $G\leq \Aut^s(X)$ if and only if $G$ embeds in $Co_1$ and there exists a primitive embedding $\mathbb{L}_G\hookrightarrow\bL_{\KTcube}$ such that \( \mathbb{L}_G\cap \mathcal{W}_{\KTcube}=\emptyset.\)
    \item There exists a hyperk\"ahler manifold $X$ of $\KTcube$--type with $G\leq \Bir^s(X)$ stable if and only if $G$ embeds in $Co_1$ and there exists a primitive embedding $\mathbb{L}_G\hookrightarrow\bL_{\KTcube}$ such that \( \mathbb{L}_G\cap \mathcal{W}_{\KTcube}^{pex}=\emptyset.\)
    \end{enumerate}
\end{prop}

We remark that by \cite[Lemma 24]{giosympaut}, a group of symplectic automorphisms of a manifold of \(\KTcube\)--type is always stable. We call any finite subgroup $G\leq O^+(\bL_{\KTcube})$ \emph{symplectic} if $(\bL_{\KTcube})_G$ is negative definite and does not contain any vector of $\mathcal{W}_{\KTcube}^{pex}$.

\subsection{Double EPW-sextics and double EPW-cubes}
In this paper $V_6$ denotes a 6-dimensional complex vector space. 
Let $\bigwedge^3V_6\times \bigwedge^3V_6\to\mathbb{C}$ be the symplectic form induced by the conformal volume form $\bigwedge^6V_6\cong \mathbb{C}$, and denote by $\LG$ the Grassmannian of Lagrangian subspaces of $\bigwedge^3V_6$ with respect to the symplectic form.  

For \([v]\in \mathbb{P}(V_6)\) consider the vector space 
\[F_{v}\defeq v\wedge \bigwedge^2 V_6\]
and denote by \(F\) the vector bundle on \(\mathbb{P}(V_6)\) whose fiber over a point \([v]\) is given by \(F_{v}\). 
Let \([A]\in\LG\) be a Lagrangian space and, for any integer \(k\geq 0\), set 
\[  Y_A[k]\defeq \{[v]\in\mathbb{P}(V_6)\mid \dim(A\cap F_{v})\geq k \}.\]
The variety \(Y_A[1]\) is called an \textit{EPW-sextic}.

Similarly, for $[U]\in\textnormal{Gr}(3, V_6)$ consider the vector space
\[ T_{U} \defeq  \bigwedge^2U\wedge V_6\]
and denote by \(T\) the vector bundle on \(\Gr\) whose fiber over a point \([U]\) is given by \(T_{U}\). 
Let \([A]\in\LG\) be a Lagrangian space and, for any integer \(k\geq 0\), set 
\[  \textnormal{Z}_A[k]\defeq \{[U]\in\Gr\;\mid\; \dim(A\cap T_{U})\geq k \}.\]
The variety \(Z_A[2]\) is called an \textit{EPW-cube}.

According to \cite{o2010double,iliev2019epw}, the following sets 
\begin{align}\label{distinguished_divisors}
\Sigma &\defeq \{[A]\in \LG\;\mid\; \mathbb{P}(A)\cap\textnormal{Gr}(3, V_6) = \emptyset\},\nonumber\\ 
\Delta &\defeq  \{[A]\in\LG\;\mid\; Y_A[3]\neq\emptyset\},\quad  \\ 
\Gamma &\defeq  \{[A]\in\LG\;\mid\; \cub{4} \neq\emptyset\},\nonumber
\end{align}
are divisors in \(\LG\) and pairwise share no common component.
\begin{prop}[\cite{o2010double,iliev2019epw,debarre2020double,rizzo}]\label{epws}
Let \([A]\in\LG\setminus\Sigma\).
\begin{enumerate}
    \item There is a canonical double cover \(\pi_A^Y\colon\widetilde{Y}_A[1]\rightarrow Y_A[1]\) branched along \(Y_A[2]\) with a natural polarization \(h_A\defeq (\pi_A^Y)^*\mathcal{O}_{Y_A[1]}(1)\), whose singular locus is given by \((\pi_A^Y)^{-1}(Y_A[3])\) and consists of a finite set of points.
    If \([A]\not\in\Delta\), then \((\widetilde{Y}_A[1], h_A)\) is a smooth polarized hyperkähler manifold of \(\KTsquare\)--type with \(h_A^2=2\) and \(\divisibility(h_A, \bL_{\KTsquare})=1\). If \([A]\in\Delta\setminus\Sigma\) then there is a symplectic projective resolution \(\widetilde{Y}_A[1]^\epsilon\rightarrow\widetilde{Y}_A[1]\) given by \(\mathbb{P}^2\)--contractions and \(\widetilde{Y}_A[1]^\epsilon\) is a smooth quasipolarized hyperk\"ahler manifold of \(\KTsquare\)--type.

    \item There is a canonical double cover \(\pi_A^Z\colon\dcub{2}\rightarrow \cub{2}\) branched along \(\cub{3}\) with a natural polarization \(H_A\defeq (\pi_A^Z)^*\mathcal{O}_{Z_A[2]}(1)\), whose singular locus is given by \((\pi_A^Z)^{-1}(\cub{4})\) and consists of a finite set of points. 
    If \([A]\not\in\Gamma\), then \((\dcub{2},H_A)\) is a smooth polarized hyperkähler manifold of \(\KTcube\)--type with \(H_A^2=4\) and \(\divisibility(H_A, \bL_{\KTcube})=2\). If \([A]\in\Gamma\setminus\Sigma\) then there is a symplectic projective resolution \(\dcub{2}^\epsilon\rightarrow\dcub{2}\) given by \(\mathbb{P}^3\)--contractions and \(\dcub{2}^\epsilon\) is a smooth quasi-polarized hyperk\"ahler manifold of \(\KTcube\)--type.
\end{enumerate}
\end{prop}
The variety \(\widetilde{Y}_A[1]\) is called a \textit{double EPW-sextic}, and the variety \(\dcub{2}\) is called a \textit{double EPW-cube}.

\subsection{Moduli of double EPW-sextics and double EPW-cubes}\label{subsec: moduli epws}

Consider the lattice \( \Lambda\defeq  \bU^{\oplus 2}\oplus \bE_8^{\oplus 2}\oplus[-2]^{\oplus2}\).
Notice that by Eichler's criterion there is only one orbit of elements of given square and divisibility in the lattices \(\bL_{\textnormal{K3}^{[2]}}\) and \(\bL_{\KTcube}\). 
Let \(h\in\bL_{\KTsquare}\) be such that \(\Homology^2=2\) and \(\divisibility(h,\bL_{\KTsquare})=1\), and let \(H\in\bL_{\KTcube}\) be such that \(\Homology^2=4\) and \(\divisibility(H,\bL_{\KTcube})=2\).
Those vectors are unique up to isometry, and their orthogonal complement is uniquely determined to be \(h^\perp \simeq H^\perp\simeq \Lambda\). 

We recall that, according to \cite{o2016moduli,iliev2019epw}, there are moduli spaces of double EPW-sextics and double EPW-cubes 
\begin{align*}
    \mathcal{M}_{sex}&\defeq \left(\LG\setminus(\Sigma\cup\Delta)\right)//\PGL(V_6),\\
    \mathcal{M}_{cub}&\defeq \left(\LG\setminus (\Sigma\cup\Gamma)\right)//\Aut(\Gr),
\end{align*}
with respective period spaces 
\begin{align*}
\Omega_{sex}&\defeq \widetilde{\bO}(\Lambda)\backslash\{x\in \Lambda_ \mathbb{C}\;\mid\;x^2=0, \;x\cdot \overline{x}>0\},\\
\Omega_{cub}&\defeq \bO^+(\Lambda)\backslash\{x\in \Lambda_ \mathbb{C}\;\mid\;x^2=0, \;x\cdot \overline{x}>0\}
\end{align*}
\noindent\cite[Section 2.3]{LO19}.
Notice that $D_\Lambda\cong (\ZZ/2\ZZ)^{\oplus2}$ and $\bO(D_\Lambda)$ has order two: we obtain that $\widetilde{\bO}(\Lambda)$ has index 2 in $\bO^+(\Lambda)$. 
There are period maps
\begin{align*}
&\mathcal{P}_{sex}:\;\mathcal{M}_{sex}\rightarrow \Omega_{sex},\\
&\mathcal{P}_{cub}:\;\mathcal{M}_{cub}\rightarrow \Omega_{cub},
\end{align*}
associating to the double EPW-sextics and the double EPW-cubes their respective Hodge structures. 

\section{Stable symplectic groups for hyperk\"ahler manifolds of \texorpdfstring{\(\KTcube\)}{K3[3]}-- type}\label{appendic lattice computation}

The strategy to determine finite groups $G$ which act by stable symplectic birational automorphisms on a manifold of \(\KTcube\)--type is to apply \autoref{wall-div_criterion}. By Torelli theorem, such groups admit a faithful orthogonal representation on $\bL_{\KTcube}$ with negative definite coinvariant sublattice. Our problem thus translates into determining finite subgroups of $\bO^+(\bL_{\KTcube})$ with negative definite coinvariant sublattices which satisfy the conditions from \autoref{wall-div_criterion}.

\begin{rmk}\label{rem monodromy}
    In regard to the classification of triples $(X, \eta, G)$ (see for instance \cite[\S3.2]{brandhorst-hofmann}), where $X$ is a manifold of $\KTcube$--type, $\eta\colon \Homology^2(X, \Z)\xrightarrow{\simeq}\bL_{\KTcube}$ is a marking  and $G\leq \Aut^s(X)$, it is customary to classify finite symplectic subgroups of $\bO^+(\bL_{\KTcube})$ up to \emph{monodromy} conjugation. In the $\KTcube$ case, we have that the associated monodromy group $\textnormal{Mon}^2(\bL_{\KTcube}) = \bO^+(\bL_{\KTcube}) \leq \bO(\bL_{\KTcube})$ is maximal \cite[Theorem 1.2]{markmanmono}. Note moreover that given $X$ a manifold of $\KTcube$--type and a subgroup $G\leq \Bir(X)$, the representation of $G$ on $\bL_{\KTcube}$ is not unique and depends on a choice of a \emph{marking} $\eta\colon \Homology^2(X, \Z)\to \bL_{\KTcube}$. ssi
\end{rmk}

According to \cite[Lemma 24]{giosympaut} and \autoref{rem monodromy}, we only have to compute a list of representatives for the $\bO^+(\bL_{\KTcube})$-conjugacy classes of finite symplectic subgroups of $\widetilde{\bO}(\bL_{\KTcube})$ which are (stably) saturated: any other finite symplectic subgroup of $\widetilde{\bO}(\bL_{\KTcube})$ is conjugate to a subgroup of one of these.

Let $1\neq G\leq \widetilde{\bO}(\bL_{\KTcube})$ be a nontrivial finite symplectic subgroup, and let us denote by $C\subseteq \bL_{\KTcube}$ the associated coinvariant sublattice. 

\begin{lem}
    If the finite symplectic subgroup $1\neq G\leq \widetilde{\bO}(\bL_{\KTcube})$ is stably saturated, then its $\bO(\bL_{\KTcube})$-conjugacy class depends only on the $\bO(\bL_{\KTcube})$-orbit of $C\subseteq  \bL_{\KTcube}$.
\end{lem}

\begin{proof}
    Since $G$ is stably saturated, we know that $G\cong \bO^\#(C)$ is uniquely determined by $C$ (\autoref{lem stab sat}). In particular, for all $f\in \bO(\bL_{\KTcube})$, we have that $fGf^{-1}$ fixes $(fC)^\perp$ and restricts to $\bO^\#(fC)$ on $fC$. Hence the result follows.
\end{proof}

By definition, $G$ fixes no nontrivial vectors in $C$, and in particular so does $\bO^\#(C)$. Moreover, \autoref{wall-div_criterion} tells us that such coinvariant sublattice $C$ embeds primitively into the Leech lattice, and since $O^\#(C)$ fixes no nontrivial vector in $C$, \autoref{lem stab sat} gives us that there exists a finite subgroup $1\neq G'\leq O(\mathbb{L})$ such that $\mathbb{L}_{G'}\simeq C$ and $G'\cong \bO^\#(C)$. Note that $\mathbb{L}^{G'}$ is nontrivial since $\rk(C)< \rk(\bL_{\KTcube}) = 23<24=\rk(\mathbb{L})$.\smallskip

Finite subgroups \(G\leq O(\mathbb{L})\) with nontrivial invariant sublattices are classified in \cite{HMLeech}, where the lattices \(\mathbb{L}_G\) are also given: we refer to the associated pairs $(\mathbb{L}_G, G)$ as \emph{Leech pairs}. Note that given a Leech pair $(\mathbb{L}_G,G)$, we can assume that $G$ is saturated in $\bO(\mathbb{L})=\bO^\#(\mathbb{L})$ meaning that $G\cong\bO^\#(\mathbb{L}_G)$ (\autoref{lem stab sat}). In order to get a complete list of representatives for the conjugacy classes of stable and stably saturated finite subgroups of symplectic birational automorphisms for hyperk\"ahler manifolds of \(\KTcube\)--type, it is equivalent to iterate the following procedure:

\begin{enumerate}
    \item choose a Leech pair $(\mathbb{L}_G, G)$,
    \item compute representatives for the orbits $\bO(\bL_{\KTcube})\cdot C$ where $\mathbb{L}_G\simeq C\subseteq \bL_{\KTcube}$ is primitive,
    \item for each such primitive sublattice $C\subseteq \bL_{\KTcube}$, determine whether \(C\cap \mathcal{W}_{\KTcube}^{pex}=\emptyset\),
    \item for each such primitive sublattice $C\subseteq \bL_{\KTcube}$ with \(C\cap \mathcal{W}_{\KTcube}^{pex}=\emptyset\), compute the finite subgroup of $\widetilde{\bO}(\bL_{\KTcube})$ which fixes $C^\perp$ pointwise and restricts to $\bO^\#(C)$ on $C$.
\end{enumerate}
For each Leech pair $(\mathbb{L}_G, G)$, one can determine representatives for the orbits in step 2 by following the usual procedure given by the proof of \cite[Proposition 1.15.1]{nikulin} (compare with \cite[Theorem 4.8]{marquand-muller}, and see \cite{brandhorst-hofmann} and references therein for computational comments).\smallskip

Checking the presence of wall divisors in a definite lattice, as in step 3 above, is a simple routine which goes back to finding \emph{short vectors} in a definite lattice, using for instance Fincke-Pohst algorithm \cite[\S2.7.3]{cohen}. 
This could be computationally demanding for vectors of large square in lattices of large rank: in order to simplify the procedure one can make use of the following.

\begin{lem}
\label{easy divisibility}
    Let $M\subseteq L$ be a primitive sublattice. The lattice $M$ contains a primitive vector \(v\) with \(v^2=d\) and \(\divisibility(v,L)=\gamma\) if and only if $M\cap \gamma L^\vee$ contains a primitive vector of square \(d\).
\end{lem}
\begin{proof}
    Let $v\in M$ be primitive. 
    Then $\divisibility(v,L) = \gamma$ if and only if $v/\gamma\in L^\vee$ is primitive, if and only if $v\in \gamma L^\vee$ is primitive. 
    Since $v\in M$ is primitive and $M\subseteq L$ is primitive, we can conclude. 
\end{proof}

Finally, we ensure that the groups constructed in step 4 lie in $\widetilde{\bO}(\bL_{\KTcube})$ thanks to \autoref{lem stab sat}, and \cite[Lemma 2.3]{gov20} which tells us that since $\bL_{\KTcube}$ has odd positive signature, any finite subgroup $G\leq \bO(\bL_{\KTcube})$ with negative definite coinvariant sublattice lies in $\bO^+(\bL_{\KTcube})$.

\begin{thm}\label{class_symplectic_groups}
    Let \(X\) be a manifold of \(\KTcube\)--type and let \(G\leq \Bir^s(X)\) be a finite stable subgroup. Then there exists a marking $\eta\colon \Homology^2(X, \mathbb{Z})\to \bL_{\KTcube}$ via which $G$ embeds into one of the groups in \cite{database} (see also \autoref{tab: lovely table 2}). Moreover, for any finite subgroup $G\leq \widetilde{\bO}(\bL_{\KTcube})$ contained in \cite{database}, there exists a manifold $X$ of $\KTcube$--type, a finite stable subgroup $G'\leq\Bir^s(X)$ and a marking $\eta$ such that $G$ is induced by $G'$ via the marking $\eta$.
\end{thm}
\begin{proof}
    This is an application of \autoref{wall-div_criterion}. We apply the procedure explained above to the list of representatives for the conjugacy classes of subgroups \(G\leq \bO(\mathbb{L})\) classified in \cite{HMLeech}. Note that $\bO(\bL_{\KTcube})/\bO^+(\bL_{\KTcube})$ is generated by the class of $-\id$ which is a central involution. This implies that $\bO^+(\bL_{\KTcube})$-conjugacy classes and $\bO(\bL_{\KTcube})$-conjugacy classes of subgroups of $\widetilde{\bO}(\bL_{\KTcube})$ coincide.
    The computations were performed with the computer algebra system OSCAR \cite{OSCAR}, which is open source. We refer to \cite{brandhorst-hofmann} and the references therein for technical details about the implementation of the procedure described in the proof of \cite[Proposition 1.15.1]{nikulin}.
    We determine exactly 219 $\bO^+(\bL_{\KTcube})$-conjugacy classes of finite saturated symplectic subgroups of $\widetilde{\bO}(\bL_{\KTcube})$: representatives for such conjugacy classes in terms of matrices are available in \cite{database}. We prescribe different notebooks and scripts for reader's convenience.
    Moreover, we collect in \autoref{sec appendix}, \autoref{tab: lovely table 2} information about each entry of our dataset. Refer to the notebooks "Example" and "Nonstable" in \cite{database} for examples of application of the procedure previously described.
\end{proof}

\begin{rmk}\label{lem stable sat to sat}
    Note that even though any group $G\leq \widetilde{\bO}(\bL_{\KTcube})$ in \cite{database} is stably saturated, it might not be saturated in $\bO^+(\bL_{\KTcube})$. If it is not, its saturation can be obtained by adding an extra nonstable isometry $g\in \bO^+(\bL_{\KTcube})$ with $g^2\in G$ (see \cite[Proposition 4.23]{marquand-muller}).
\end{rmk}

As an application, we determine the groups with coinvariant sublattice of rank \(20\) (maximal) and which fix a polarization of the same type of the polarization of a double EPW-cube.
\begin{cor}    
\label{lattices_of_examples}
Let \(X\) be an hyperk\"ahler manifold of \(\KTcube\)--type with $H\in\NS(X)$ of numerical type $(4,2)$, and let $\Bir^s_H(X)$ be the finite group of symplectic birational automorphisms preserving $H$.
If there exists a stable and stably saturated finite subgroup \(G\leq\Bir^s_H(X)\) with \(\rk(\Homology^2(X, \mathbb{Z})_{G})=20\), then $G$ is given in \autoref{table max pol}. The genus of \(\NS(X)\) and the associated transcendental lattices \(\T(X)\) are also displayed.
\end{cor}
\begin{proof}
    In this case \(\NS(X)\) has signature \((1,20)\) and \(T(X)\) has signature \((2,0)\). For any group \(G\) as in \autoref{class_symplectic_groups} with \(\rk(\Homology^2(X, \mathbb{Z})_{G})=20\) we use \autoref{easy divisibility} and decide whether $\Homology^2(X, \mathbb{Z})^{G}\subseteq\Homology^2(X, \mathbb{Z})$ has a $G$-invariant vector $H$ of type \((4,2)\) (divisibility in $\Homology^2(X, \ZZ)$). 
    For each orbit of such vectors, the corresponding transcendental lattice is given by $H^\perp \subset \Homology^2(X, \mathbb{Z})^{G}$. For each entry in \autoref{table max pol}, we refer the Id of the corresponding conjugacy class of representations of $G$ on \(\bL_{\KTcube}\) as given in \autoref{tab: lovely table 2}.
\end{proof}
{
\renewcommand\arraystretch{1.2}
\setlength{\tabcolsep}{2pt}
\begin{table}[ht]

\caption{Stable and stably saturated $(4,2)$-polarized actions with maximal coinvariant sublattice}\label{table max pol}\vspace*{0.3cm}
\centering
    \small{\begin{tabular}{ccccc|ccccc}

 Id&$G$& Regular & $\textnormal{T}(X)$&$g(\NS(X))$&Id& $G$ & Regular & $\textnormal{T}(X)$&$g(\NS(X))$ \\
 \hline
 \cellcolor{lightgray!40!white}102d&\cellcolor{lightgray!40!white}$L_3(4)$&\cellcolor{lightgray!40!white} No&\cellcolor{lightgray!40!white}$\begin{pmatrix}12&0\\0&28\end{pmatrix}$&\cellcolor{lightgray!40!white} $\II_{(1,21)}4^{-1}_53^{-1}7^{-1}$&\cellcolor{lightgray!40!white}124&\cellcolor{lightgray!40!white}$[384, 18134]$&\cellcolor{lightgray!40!white}No&\cellcolor{lightgray!40!white}$\begin{pmatrix}2&0\\0&16\end{pmatrix}$&\cellcolor{lightgray!40!white} $\II_{(1,21)}2^{1}_74_3^{-1}16_3^{-1}$\\

  106b&$C_2^4\rtimes \mathcal{A}_6$& No& $\begin{pmatrix}4&0\\0&24\end{pmatrix}$&$\II_{(1,21)}2^28^{-1}_33^1$&131 &$[192, 1494]$&No&$\begin{pmatrix}8&0\\0&8\end{pmatrix}$&$\II_{(1,21)}4^{-1}_58^{-2}_4$\\
 
  \cellcolor{lightgray!40!white}108a&\cellcolor{lightgray!40!white}$\mathcal{A}_7$&\cellcolor{lightgray!40!white}Yes& \cellcolor{lightgray!40!white} $\begin{pmatrix}6&0\\0&70\end{pmatrix}$&\cellcolor{lightgray!40!white}$\II_{(1,21)}4^1_13^15^17^1$&\cellcolor{lightgray!40!white}133b&\cellcolor{lightgray!40!white}$C_2\times M_9$&\cellcolor{lightgray!40!white}No&\cellcolor{lightgray!40!white}$\begin{pmatrix}4&-2\\-2&10\end{pmatrix}$&\cellcolor{lightgray!40!white}$\II_{(1,21)}2^{-3}_44^{-1}_59^{-1}$\\
  
  119a&$M_{10}$&Yes& $\begin{pmatrix}4&0\\0&30\end{pmatrix}$&$\II_{(1,21)}2^{-1}_34^2_03^{-1}5^1$&133b&$C_2\times M_9$&No&$\begin{pmatrix}4&0\\0&36\end{pmatrix}$&$\II_{(1,21)}2^{-3}_44^{-1}_59^{-1}$\\

 \cellcolor{lightgray!40!white}119e&\cellcolor{lightgray!40!white} $M_{10}$&\cellcolor{lightgray!40!white}No&\cellcolor{lightgray!40!white}$\begin{pmatrix}4&0\\0&30\end{pmatrix}$&\cellcolor{lightgray!40!white}$\II_{(1,21)}2^3_73^{-1}5^1$&\cellcolor{lightgray!40!white}137a &\cellcolor{lightgray!40!white}$\mathcal{S}_5$&\cellcolor{lightgray!40!white}No&\cellcolor{lightgray!40!white}$\begin{pmatrix}6&0\\0&10\end{pmatrix}$&\cellcolor{lightgray!40!white}$\II_{(1,21)}2^{-2}_24^1_13^15^{-1}$\\

  120a&$L_2(11)$&Yes& $\begin{pmatrix}22&0\\0&22\end{pmatrix}$&$\II_{(1,21)}4^1_111^2$&149b&$C_2\times F_5$&No&$\begin{pmatrix}10&0\\0&10\end{pmatrix}$&$\II_{(1,21)}2^{-2}_24^1_75^2$\\
 \hline
\end{tabular}}
\end{table}}
  We observe that \autoref{class_symplectic_groups} gives a complete classification of saturated groups of symplectic regular automorphisms for hyperk\"ahler manifolds of \(\KTcube\)--type, since regular symplectic automorphisms are stable \cite[Lemma 24]{giosympaut}. However, birational symplectic automorphisms might not always be stable, as shown in the following example (see also \autoref{lem stable sat to sat}).
Because of this limitation, the use of Leech pairs is not enough for the classification of finite subgroups of symplectic birational automorphisms.

\begin{ex}
    There exist nonstable groups of symplectic birational automorphisms of $\KTcube$--type hyperk\"ahler manifolds.
    For instance, consider the pair $(\mathbb{L}_G, G)$ no. 77 from H\"ohn--Mason database \cite{HMLeech}: the group $G$ is actually isomorphic to the finite simple group $L_2(7)$. 
    The associated coinvariant sublattice $\mathbb{L}_G$ has rank 19 and $\bL_{\KTcube}$ admits exactly two $\bO(\bL_{\KTcube})$-orbits of primitive sublattices $C\simeq \mathbb{L}_G$ with \(C\cap\mathcal{W}^{pex}_{\KTcube}=\emptyset\) (see \autoref{tab: lovely table 2}, entries 77a and 77b). 
    One of such $C$'s is so that $C\cap \mathcal{W}_{\KTcube}=\emptyset$ meaning that, according to \autoref{wall-div_criterion}, $\bO^\#(C)$ can be realized as a finite group of symplectic automorphisms on a manifold of $\KTcube$--type. Let us fix such primitive sublattice $C$, and let again $G\leq \widetilde{\bO}(\bL_{\KTcube})$ be defined as the identity on $F := C^\perp$ and restricts to $\bO^\#(C)$ on $C$.
    The invariant sublattice $F =\bL_{\KTcube}^G$ is isometric to 
    \[\scriptstyle{\begin{pmatrix}
        2&1&0&0\\1&4&0&0\\0&0&-4&0\\0&0&0&28
    \end{pmatrix}}.\]
    Following the ideas of the extension approach of \cite[\S5]{brandhorst-hashimoto}, later extended by \cite[Theorem 3.25]{brandhorst-hofmann}, adapted by \cite{comparin-demelle,wawak2022very} for hyperk\"ahler manifolds of type $\KTsquare$ and reformulated for symplectic actions in \cite[\S5]{marquand-muller}, one can find that $F$ admits an involution with negative definite coinvariant sublattice which can be extended to an isometry $g\in \bO^+(\bL_{\KTcube})$ whose square lies in $G$. 
    Such an isometry is symplectic and nonstable.
    Hence, the group $G' \defeq  \langle G, g\rangle$ is symplectic and nonstable.
    One has moreover
    \[\bL_{\KTcube}^{G'}\simeq \scriptstyle{\begin{pmatrix}
        2&0&0\\0&6&2\\0&2&10
    \end{pmatrix}}\]
    and the associated coinvariant sublattice lies in the genus $\II_{(0, 20)}4^2_27^1$\footnote{see the notebook "Nonstable" in \cite{database} for more computational details}. One can realize such an example explicitly as a birational model of the Hilbert scheme of three points on the K3 surface of degree 2 described in \cite[\S 6.6, 74a]{brandhorst-hashimoto}. See \cite[Section 6.4]{kapustka2022epw} for a thorough discussion on the geometry of Hilbert cubes of K3 surfaces of degree 2.
\end{ex}

\section{The automorphism group of a double EPW-cube}\label{sec describe aut}
In this section we describe the automorphism group of a double EPW-cube associatred to a  Lagrangian subspace \([A]\not\in(\Sigma\cup\Gamma)\) which satisfies some further generality assumption. The strategy followed is analogous to that of the case of double EPW-sextics treated in \cite[Proposition A.2]{debarre2022gushel}. \bigskip

Recall that we have an isomorphism \(\Aut(\Gr)\cong \PGL(V_6)\times \langle \delta\rangle\), where \(\delta\) is the involution that sends a \(3\)-space to its dual \(3\)-space with respect to the symplectic form. 
The latter induces an involution that we will call again \(\delta\in\Aut(\LG)\). 
This gives a map 
\[p:\mathcal{M}_{sex}\dashrightarrow\mathcal{M}_{cub}\] 
which is a double cover ramified on the locus of double EPW-sextics associated with a self-dual Lagrangian space (see \cite{o2008dual,o2015periods}), sending the isomorphism class of \(\dsex{1}\) to the isomorphism class of \(\dcub{2}\). 
According to \cite[Corollary 6.1]{kapustka2022epw}, there is a commutative diagram 
    \begin{equation}\label{moduli square}
\begin{tikzcd}
\mathcal{M}_{sex} \arrow[r, dashed,"p"] \arrow["\mathcal{P}_{sex}"',d] & \mathcal{M}_{cub} \arrow[d,"\mathcal{P}_{cub}"] \\
\Omega_{sex} \arrow[r, dashed,"i"]                        & \Omega_{cub}               
\end{tikzcd}
    \end{equation}
    where \(i:\Omega_{sex}\rightarrow \Omega_{cub}\) is the quotient map given by the inclusion \(\widetilde{\bO}(\Lambda)\leq \bO^+(\Lambda)\), where we recall that $\Lambda := \bU^{\oplus2}\oplus \bE_8^{\oplus2}\oplus [-2]^{\oplus2}$.
\begin{lem}\label{two_to_one}
     Let \([A_1],[A_2]\in\LG\setminus(\Sigma\cup\Gamma)\). 
     Then \(\vcub{A_1}{2}\) and \(\vcub{A_2}{2}\) are $\PGL(\bigwedge^3V_6)$-isomorphic if and only if \([A_1]\) and \([A_2]\) are in the same \(\Aut(\Gr)\)-orbit.
\end{lem}
\begin{proof}
    One direction is clear: if $[A_1]$ and $[A_2]$ are in the same orbit, since \(\Aut(\Gr)\leq \PGL(\bigwedge^3V_6)\), we obtain that the associated EPW-cubes are linearly equivalent.
    
    Now suppose that \(\vcub{A_1}{2}\) and \(\vcub{A_2}{2}\) are linearly equivalent via a certain \(f\in\PGL(\bigwedge^3V_6)\). Since the double covers \(\vdcub{A_i}{2}\to \vcub{A_i}{2}\) are natural, we have that \(\vdcub{A_1}{2}\) and \( \vdcub{A_2}{2}\) are isomorphic as polarized manifolds. By the proof of \cite[Proposition 5.1]{iliev2019epw}, the latter actually implies that there exists \( g\in\Aut(\Gr)\) such that \(g([A_1])=[A_2]\). 
    \end{proof}
    
\begin{rmk}
    This result translates at the level of moduli spaces. In fact, from the description \(\Aut(\Gr)\cong \PGL(V_6)\times\langle\delta\rangle\), the previous statement can be seen as a consequence of the fact that the map \(p\) in \autoref{moduli square} is just the quotient map via \(\delta\).
\end{rmk}

For any subgroup $G\leq\Aut(\Gr))$ and for any Lagrangian space $[A]\in \LG$, let us denote by $G_A$ the subgroup of automorphisms fixing $[A]$.
Similarly to the case of EPW-sextics, we have:
\begin{prop}\label{automorphi_are_linear}
   For $[A]\in \LG\setminus(\Sigma\cup\Gamma)$, we have an isomorphism  \[\textnormal{Stab}_{\PGL(\bigwedge^3V_6)}(\cub{2})\cong \Aut(\Gr)_A
.\]
\end{prop}
\begin{proof}
Clearly, any automorphism $g\in\Aut(\Gr)\leq \PGL(\bigwedge^3V_6)$ fixing $[A]\in \LG\setminus(\Sigma\cup\Gamma)$ also preserves $\cub{2}$. Suppose now \(g\in\textnormal{Stab}_{\PGL(\bigwedge^3V_6)}(\cub{2})\), then \(g\) induces an automorphism of $\Gr$ according to \cite[Lemma 5.2]{iliev2019epw}. 
We conclude using \autoref{two_to_one}.
\end{proof}

In the following we adapt the proof of \cite[Proposition A.2]{debarre2022gushel} to the case of EPW-cubes.
\begin{prop}\label{automorphisms_double_cover}
Let \([A]\in\LG\setminus(\Sigma\cup\Gamma)\) and consider \(\pi\colon\dcub{2}\rightarrow\cub{2}\) the associated double EPW-cube with polarization \(H=\pi^*\mathcal{O}_{\cub{2}}(1)\). Let us denote by $\iota$ the covering involution, and by \(\Aut_H(\dcub{2})\) is the group of automorphisms preserving \(H\). If \(A\) and \(\delta(A)\) are not in the same \(\PGL(V_6)\)-orbit, then there is an isomorphism

    \[\Aut_H(\dcub{2})\cong \PGL(V_6)_A\times \langle \iota\rangle\]
where the group \(\PGL(V_6)_A\) corresponds to the subgroup \(\Aut_H^s(\dcub{2})\) of symplectic automorphisms preserving \(H\).
\end{prop}
\begin{proof}
Since \([A]\not\in \Sigma\), by \cite[Theorem 4.2, Theorem 5.7]{debarre2020double} we have \( \dcub{2}=\Spec(\mathcal{O}_{\cub{2}}\oplus \mathcal{R}_2(-2) )\)
 , where \(\mathcal{R}_2=(\bigwedge^2 \mathcal{C}_2)^{\vee\vee}\) and \(\mathcal{C}_2=\Coker(T\rightarrow A^\vee\otimes\mathcal{O}_{\Gr})_{\mid \cub{2}}\).
Any element \(g\in\Aut_H(\dcub{2})\) induces an automorphism of \(\mathbb{P}(\Homology^0(\dcub{2},H)^\vee)\cong\mathbb{P}(\Homology^0(\cub{2},\mathcal{O}_{\cub{2}}(1))^\vee)\cong\mathbb{P}(\bigwedge^3V_6) \) which preserves the locus \(\cub{2}\). Since \(A\) and \(\delta(A)\) are not in the same \(\PGL(V_6)\)-orbit, we have \(\Aut(\Gr)_A=\PGL(V_6)_A\). Every element of \(\PGL(V_6)_A\) preserves the sheaf \(\mathcal{O}_{\cub{2}}\oplus \mathcal{R}_2(-2)\) and its algebra structure, so it induces an automorphism of \(\dcub{2}\). According to \autoref{automorphi_are_linear}, we get a central extension
\begin{equation}\label{ses cubes}
1\to \langle\iota\rangle \to\Aut_H(\dcub{2}) \rightarrow \PGL(V_6)_A\to 1.
\end{equation}
There is another central extension 
\[1\rightarrow\mu_6\rightarrow \widetilde{G} \rightarrow  \PGL(V_6)_A\rightarrow 1,\] where \(\widetilde{G}\) is the preimage of \(\PGL(V_6)_A\) via the map \( \SL(V_6) \rightarrow \PGL(V_6)\).
Following the ideas of \cite[Appendix A.1]{debarre2022gushel}, we have that the induced action of $\widetilde{G}$ on $\bigwedge^3V_6$ factors through the quotient $\widetilde{G}/\mu_3$, and according to \cite[Lemma A.1]{debarre2022gushel} we have that $\widetilde{G}/\mu_3$ embeds into $\GL(A)$. By similar arguments as in the proof of \cite[Lemma A.2]{debarre2022gushel}, and according to the description of $\mathcal{C}_2$ given in \cite[Theorem 4.2]{debarre2020double}, we observe that the group $\widetilde{G}/\mu_3$ acts on $\dcub{2}$ preserving $H$. Moreover, since $\mu_2\leq \widetilde{G}/\mu_3$ acts trivially on both $\mathcal{R}_2$ and $\mathcal{O}_{\cub{2}}(-2)$, we get an injective morphism
\[\psi\colon \PGL(V_6)_A\hookrightarrow \Aut_H(\dcub{2})\]
which splits the short exact sequence in \autoref{ses cubes}.

The action of \(\Aut_H(\dcub{2})\) on \( \Homology^2(\dcub{2},\mathcal{O}_{\dcub{2}})\cong \mathbb{C}\sigma_{\dcub{2}}\) determines a character \(\Aut_H(\dcub{2})\rightarrow \mathbb{C}^\times\), with finite (cyclic) image of order \(r\geq 2\), that sends \(\iota\) to \(-1\). Note that $\PGL(V_6)_A$ acts trivially on \( \Homology^2(\dcub{2},\mathcal{O}_{\dcub{2}})\cong \mathbb{C}\sigma_{\dcub{2}}\) since \(\PGL(V_6)\) has no nontrivial characters. In conclusion, \(\psi\) is a section of \(\Aut_H(\dcub{2}) \rightarrow \PGL(V_6)_A\) and we have an insomorphism $\Aut_H(\dcub{2})\cong \langle\iota\rangle\times \PGL(V_6)_A$ where $\PGL(V_6)_A$ corresponds to $\Aut_H^s(\dcub{2})$.
\qedhere
\end{proof}

\section{Birational automorphisms of general double EPW-cubes}\label{sec general bir}
We compute the full group of birational automorphisms of the general smooth double EPW-cubes, and of the desingularization of the general singular double EPW-cube in the family described by $\Gamma$.\bigskip

The family of smooth double EPW-cubes described in \cite{iliev2019epw} is locally complete in moduli, so we have that for a general $[A]\notin (\Gamma\cup\Sigma)$, the associated double EPW-cube $\dcub{2}$ has Picard rank 1. In fact, its N\'eron--Severi lattice is spanned by the class of the polarization $H$, of numerical type $(4, 2)$. In particular,
\[\T(\dcub{2}) \simeq  \bU^{\oplus 2}\oplus \bE_8^{\oplus 2}\oplus[-2]^{\oplus2}\]
(see also \autoref{subsec: moduli epws}). The following is actually a special case of a more general result \cite[Proposition 4.3]{debarre2022hyper}.

\begin{prop}\label{bir general smooth cube}
    Let $[A]\notin (\Gamma\cup\Sigma)$ be general. Then
    \[\Bir(\dcub{2}) = \Aut_{H}(\dcub{2}) = \langle\iota\rangle,\]
    where $\iota$ is the covering involution. In particular, 
    \[\Bir^s(\dcub{2}) = \Aut^s(\dcub{2}) = \{\id\}.\]
\end{prop}

\begin{proof}
    Any birational automorphism of $\dcub{2}$ induces an isometry in $\bO^+(\NS(\dcub{2}))$: since the N\'eron-Severi lattice of $\dcub{2}$ is positive definite of rank 1, we have that any birational automorphism of $\dcub{2}$ acts trivially on  $\NS(\dcub{2})$. In particular, we already observe that $\Bir(\dcub{2}) = \Aut_H(\dcub{2})$ and that $\Bir^s(\dcub{2})$ is trivial. Together, we infer that $\Bir(\dcub{2})$ is cyclic, generated by a regular purely nonsymplectic automorphism. We conclude using the description of $\Aut_H(\dcub{2})$ given in \autoref{automorphisms_double_cover}.
\end{proof}

\begin{rmk}
    For $[A]\in \LG$ general, it is already known that $\PGL(V_6)_A$ is trivial \cite[Proposition B.9]{dk20}.
\end{rmk}

We show in the rest of this section that we observe a similar result for the desingularization of the double EPW-cube for a general $[A]\in\Gamma$.\smallskip

According to \cite[Lemma 3.2]{rizzo}, if $[A]\in\Gamma\setminus\Sigma$ is a general Lagrangian, then $\cub{4}$ consists of a unique (smooth) point which is the only singular point of $\dcub{2}$. Hence, following \cite[Propositon 5.1, Lemma 5.2, Theorem 5.3]{rizzo} we have that $\dcub{2}$ admits two smooth projective small resolutions \(\dcub{2}^{\epsilon_1}\rightarrow\dcub{2}\) and \(\dcub{2}^{\epsilon_2}\rightarrow\dcub{2}\) which are hyperk\"ahler manifolds of $\KTcube$--type. These two manifolds are quasipolarized by a big and nef class of numerical type $(4, 2)$, and they are related by the flop of a $\mathbb{P}^3$.

Since the family of resolutions of singular double EPW-cubes parametrized by $\Gamma\setminus\Sigma$ is 19-dimensional we have that for a general $[A]\in \Gamma$, the Picard rank of $\dcub{2}^{\epsilon_1}$ (and $\dcub{2}^{\epsilon_2}$) is at most 2.

\begin{lem}\label{lem ns general A gamma}
    Let \([A]\in\Gamma\) be general. We have \[\NS(\dcub{2}^{\epsilon_1})\simeq \begin{pmatrix}
    4&2\\2&-2
\end{pmatrix}\] and \[\textnormal{T}(\dcub{2}^{\epsilon_1})\simeq \bU^{\oplus2}\oplus \bA_2\oplus \bD_7\oplus \bE_8.\]
\end{lem}
\begin{proof}
The quasipolarization of numerical type $(4, 2)$ of $\dcub{2}^{\epsilon_1}$ gives a class \(H\in\NS(\dcub{2}^{\epsilon_1})\). 
Now, by \cite[Theorem 5.3]{rizzo}, we know that the map $\pi^{\epsilon_1}\colon\dcub{2}^{\epsilon_1}\rightarrow\dcub{2}$ is a small resolution that contracts a $\mathbb{P}^3$ to the singular point of $\dcub{2}$.
According to \cite{hassett2010intersection}, this corresponds to the existence of a class
\(D\in\NS(\dcub{2}^{\epsilon_1})\) of type $(-12, 2)$ and which is orthogonal to \(H\).
This implies that \(\NS(\dcub{2}^{\epsilon_1})\) has rank 2 and we have embeddings $\langle H, D\rangle\subset \NS(\bcub{2}^\epsilon)\subset \bL_{\KTcube}$, where the first is of finite index and the second is primitive.
Observe that since \(\divisibility(H,\bL_{\KTcube})=\divisibility(D,\bL_{\KTcube})=2\) and the discriminant group of $\bL_{\KTcube}$ has only one element of order \(2\), we have that \(\langle H, (H+D)/2\rangle\) is an even overlattice of $\langle H, D\rangle$. The lattice \(\langle H, (H+D)/2\rangle\) has determinant $-12$, and none of its overlattices are even: we deduce that \(\langle H, (H+D)/2\rangle\) is actually primitive in \(\bL_{\KTcube}\). 
In conclusion, we have \(\NS(\dcub{2}^{\epsilon_1})\simeq \langle H, (H+D)/2\rangle\), which has the wanted Gram matrix. 
Finally, one can check that the embedding \(\NS(\dcub{2}^{\epsilon_1})\subset \bL_{\KTcube}\) is unique up to the action of $\bO(\bL_{\KTcube})$ and that its orthogonal complement is isometric to $\bU^{\oplus2}\oplus \bA_2\oplus \bD_7\oplus \bE_8$.
\end{proof}

We give an explicit description of the movable cone $\Mov(\dcub{2}^{\epsilon_1})$ for a general Lagrangian space \([A]\in\Gamma\). 
Classes that are dual to extremal rays of the Mori cone of a hyperk\"ahler manifold of \(\KTcube\)--type are numerically characterized and called wall divisors, see \autoref{wall_divisors}. The possible squares of wall divisors for hyperk\"ahler manifolds of \(\KTcube\)--type are \(-2,-4,-12\) or \(-36\). We see from \autoref{lem ns general A gamma} that \(\NS(\dcub{2}^{\epsilon_1})\) has no isotropic vectors nor primitive vectors of square $-4$ or $-36$. The only possibility left are vectors of type \((-2,1)\) or \((-12,2)\).

Recall that the movable cone of \(\dcub{2}^{\epsilon_1}\) is the closure of a fundamental domain which contains the class $H$, for the action of the so-called \emph{Weyl group} $W\leq \bO^+(\NS(\dcub{2}^{\epsilon_1})\otimes\mathbb{R})$ on the positive cone of $\dcub{2}^{\epsilon_1}$. The latter group $W$ is generated by the reflections in integral prime exceptional divisors of \(\dcub{2}^{\epsilon_1}\). Since the latter prime exceptional divisors are necessarily of square $-2$, we can use an algorithm of Vinberg \cite{vin75}\footnote{see notebook "NSgeneral" of \cite{database} for how we use it on the Oscar sytem \cite{OSCAR}} to find a set of such vectors whose orthogonal complements bound the so-called fundamental exceptional chamber $\mathcal{FE}_{\dcub{2}^{\epsilon_1}}$ of $\dcub{2}^{\epsilon_1}$. By doing so, one determines that such a set of prime exceptional divisors consists of
  \[(H-D)/2\quad \text{ and }\quad  (H+D)/2,\]
and their respective orthogonal complements are spanned by
\[(3H-D)/2\quad \text{ and }\quad  (3H+D)/2\]
which have both type \((6,1)\). Similarly, we find that the K\"ahler cone of \(\dcub{2}^{\epsilon_1}\) is one of the two chambers obtained by cutting $\mathcal{FE}_{\dcub{2}^{\epsilon_1}}$ with $\mathbb{R}H = D^\perp$. The situation being symmetric for $\dcub{2}^{\epsilon_1}$ and $\dcub{2}^{\epsilon_2}$, related by a flop which is birational but nonregular, we obtain that the other chamber is the pullback by $f$ of $\mathcal{K}_{\dcub{2}^{\epsilon_2}}$. 
We draw in \autoref{fig1} the birational K\"ahler cone of \(\dcub{2}^{\epsilon_1}\) which is the disjoint union of two K\"ahler-type chambers being the K\"ahler cone \(\mathcal{K}_{\dcub{2}^{\epsilon_1}}\) of \(\dcub{2}^{\epsilon_1}\) and the pullback \(f^\ast\mathcal{K}_{\dcub{2}^{\epsilon_2}}\) of the K\"ahler cone of \(\dcub{2}^{\epsilon_2}\) by the flop \(f\colon\dcub{2}^{\epsilon_1}\dashrightarrow\dcub{2}^{\epsilon_2}\).

\begin{lem}
    The two projective resolutions $\dcub{2}^{\epsilon_1}$ and $\dcub{2}^{\epsilon_2}$ are isomorphic as projective varieties.
\end{lem}

\begin{proof}
    The covering involution \(\iota\) on $\dcub{2}$ induces a birational automorphism $\hat{\iota}$ of $\dcub{2}^{\epsilon_1}$ whose action on cohomology $\hat{\iota}^\ast$ fixes $H$ and is negative identity on the orthogonal complement. This is a nonsymplectic involution which is nonregular, as the action on $\NS(\dcub{2}^{\epsilon_1})$ given by the reflection $\tau_D$ in \autoref{fig1} does not preserve $\mathcal{K}_{\dcub{2}^{\epsilon_1}}$. But, if we denote by $f^\ast\colon \Homology^2(\dcub{2}^{\epsilon_2}, \mathbb{Z})\to \Homology^2(\dcub{2}^{\epsilon_1}, \mathbb{Z})$ the Hodge isometry induced by the flop $f$, then the isometry $\hat{\iota}^\ast\circ f^\ast$ sends $\mathcal{K}_{\dcub{2}^{\epsilon_2}}$ to $\mathcal{K}_{\dcub{2}^{\epsilon_1}}$. By Torelli theorem, this implies that the two resolutions are actually isomorphic.
\end{proof}

\begin{notation}
    For $[A]\in \Gamma\setminus\Sigma\subset \LG$ general, we write $\dcub{2}^{\epsilon} := \dcub{2}^{\epsilon_1}\cong \dcub{2}^{\epsilon_2}$ for the smooth projective hyperk\"ahler desingularization of $\dcub{2}$.
\end{notation}

\begin{figure}[t]
    \caption{Movable cone of $\dcub{2}^{\epsilon_1}$}\vspace*{0.3cm}
    \centering
    \includegraphics[width=10cm]{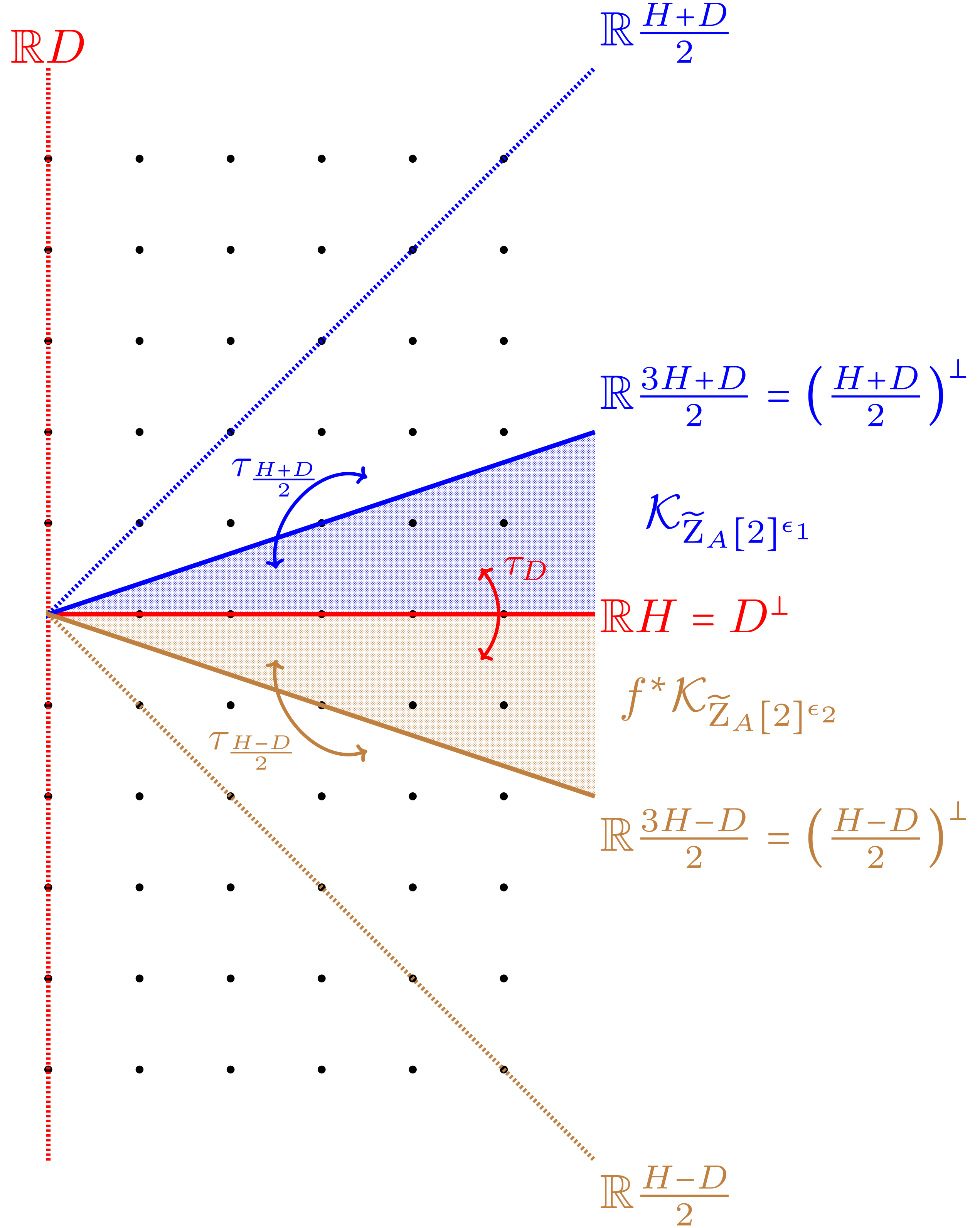}
    \label{fig1}
\end{figure}

\begin{prop}
     Let \([A]\in\Gamma\) be general. Then 
     \[\Bir(\dcub{2}^{\epsilon}) = \Bir_H(\dcub{2}^{\epsilon}) = \langle \hat{\iota}\rangle.\]
     In particular, \[\Bir^s(\dcub{2}^{\epsilon}) = \Aut(\dcub{2}^{\epsilon}) = \{\id\}.\]
\end{prop}

\begin{proof}
    Since the transcendental lattice $\T(\dcub{2}^{\epsilon})$ of the projective manifold $\dcub{2}^{\epsilon}$ has odd rank, any birational automorphism of $\dcub{2}^{\epsilon}$ is either symplectic or antisymplectic, i.e. acts as $\pm\id$ on $\T(\dcub{2}^{\epsilon})$. Moreover, any birational automorphism of $\dcub{2}^{\epsilon}$ restricts to an isometry in $\bO^+(\NS(\dcub{2}^{\epsilon}))$ (because they must preserve the positive cone in $\NS(\dcub{2}^{\epsilon})\otimes \mathbb{R}$). According to \cite[Remark 4.27]{brandhorst-hofmann}, we have that $\bO(\NS(\dcub{2}^{\epsilon})) = \bO^+(\NS(\dcub{2}^{\epsilon}))\times\{\pm \id\}$, where the first factor is generated by the three reflections $\tau_D$, $\tau_{(H+D)/2}$ and $\tau_{(H-D)/2}$ \footnote{see the notebook "NSgeneral" of \cite{database} for a computational proof of this statement}. Now, any birational automorphism of $\dcub{2}^{\epsilon}$ must preserve the fundamental exceptional chamber of $\dcub{2}^{\epsilon}$. To support the rest of the proof, we represent in \autoref{fig1} the previously mentioned isometries of $\NS(\dcub{2}^{\epsilon})$.  Since $\tau_{(H+D)/2}\tau_D = \tau_D\tau_{(H-D)/2}$, it follows that any element of $\bO^+(\NS(\dcub{2}^{\epsilon}))$ is of the form $\tau_D^i\alpha$ where $i=0,1$ and $\alpha$ is a finite word in $\tau_{(H+D)/2}$ and $\tau_{(H-D)/2}$. While $\tau_D$ preserves $\mathcal{FE}_{\dcub{2}^{\epsilon}}$, any nontrivial word of the latter form maps $H$ outside of $\mathcal{FE}_{\dcub{2}^{\epsilon}}$ (see \autoref{fig1}). Therefore, it follows that the action of any birational automorphism must coincide with $\tau_D$ or $\id$ on $\NS(\dcub{2}^{\epsilon})$. On the one hand, since $\tau_D$ is not stable, we can only extend it by negative identity on $\T(\dcub{2}^{\epsilon})$: the resulting isometry coincides with $\hat{\iota}^\ast$.  On the other hand, the identity on $\NS(\dcub{2}^{\epsilon})$ can only be extended by the identity on $\T(\dcub{2}^{\epsilon})$, giving rise to the identity of $\Homology^2(\dcub{2}^{\epsilon}, \mathbb{Z})$. By Torelli theorem, this implies that $\dcub{2}^{\epsilon}$ has no nontrivial symplectic birational automorphisms, and $\Bir(\dcub{2}^{\epsilon})$ has order 2 and it is generated by $\hat{\iota}$.
\end{proof}

\begin{rmk}
According to \cite[Table H3]{hassett2010intersection}, the linear system corresponding to the class \((H+D)/2\) determines a divisorial contraction \(\kappa\colon \dcub{2}^{\epsilon}\to B\). The associated contracted divisor $D'$ is represented by \((3H+D)/2\) in \(\NS(\dcub{2}^{\epsilon})\) (see \autoref{fig1}) and the image $\kappa(D')$ is a hyperk\"ahler fourfold $Y$ of type $\textnormal{K3}^{[2]}$. Moreover, there exists a dense open subset $Y_o\subseteq Y$ over which $\kappa$ defines a $\mathbb{P}^1$-bundle.
\end{rmk}

Finally, we conclude with a consequence of \autoref{lem ns general A gamma}, which has also been proved in \cite[Proposition 6.3]{rizzo} from a period point of view.

\begin{cor}\label{not_moduli_space}
    For a general \([A]\in\Gamma\), there exists no K3 surface $S$ such that \(\dcub{2}^{\epsilon}\) is birational a moduli space of twisted sheaves on $S$.
\end{cor}

\begin{proof}
Consider the unique primitive embedding of $\bL_{\KTcube}$ into the {\em Mukai lattice} $\Lambda_{24}\defeq  \bU^{\oplus4}\oplus \bE_8^{\oplus2}$, with orthogonal complement $\langle v\rangle\simeq [4]$ corresponding to a Mukai vector of square 4. 
The natural pure Hodge structure of weight 2 on \(\Lambda_{24}\) is so that the \((1,1)\)-part \(\Lambda_{24}^{1,1}\) is given by the primitive closure of $\NS(\dcub{2}^{\epsilon})\oplus \langle v\rangle$ in $\Lambda_{24}$. 
From the description of \autoref{lem ns general A gamma}, we deduce that it coincides with the lattice \[N \defeq  {\scriptscriptstyle{\begin{pmatrix}4 &2 &-2\\ 2 &-2& -1\\ -2 &-1 &2\end{pmatrix}}}=\left\langle H,(H+D)/2, (v-H)/2\right\rangle,\] which is an index \(2\) overlattice of $\NS(\dcub{2}^{\epsilon})\oplus \langle v\rangle$. 
It is known from \cite[Proposition 3.1]{camere2019verra} that $\dcub{2}^{\epsilon}$ is birational to moduli space of twisted sheaves on a K3 surface if and only if $N$ contains a copy of $\bU(k)$ for some $k\geq 1$.
However, according to \cite[Chapter 4, Lemma 2.5]{cassels} the quadratic space $N\otimes_\mathbb{Z}\mathbb{Q}_2$
does not contain any isotropic vector, hence $N$ cannot contain any rescaled copy of $\bU$.
Thus, we conclude that $\dcub{2}^{\epsilon}$ is not birational to a moduli space of twisted sheaves on any K3 surface.
\end{proof}

\section{Very symmetric EPW-cubes}\label{sec very sym}

In this section we exhibit examples of double EPW-cubes with a group of symplectic birational automorphisms whose action on cohomology has invariant sublattice of minimal rank. The groups we consider are canonically polarized (and appear in the classification \cite{Hohn2014FiniteGO}): we give a lattice-theoretical description of the action.\bigskip

In the following we present Lagrangian spaces \([A]\in\LG\) that are invariant under the action of a group $G$ from \autoref{group_actions}, and we study properties of the associated double EPW-cubes $\dcub{2}$. In particular, by the description of the automorphism group of $\dcub{2}$ (see \autoref{automorphi_are_linear} and \autoref{automorphisms_double_cover}), we obtain that $G$ acts symplectically on $\dcub{2}$.

For the case \(G=L_2(11)\), the Lagrangian space we consider is the Klein Lagrangian introduced in \cite{debarre2022gushel}.
For the remaining cases we use the following construction, which is inspired by \cite[Section 2]{billi2022double} where the case of the group \(\mathcal{A}_7\) was already treated.

 According to \cite{Wilson1985ATLASOF}, for \(G=\mathcal{A}_7,M_{10},L_3(4)\) there exists a group \(\widetilde{G}\) and a nontrivial element \(\eta\in\widetilde{G}\) such that $\widetilde{G}/\langle \eta\rangle\cong G$, and there exists an irreducible representation $\rho\colon \widetilde{G}\rightarrow \GL (V_6)$ whose center is generated by \(\eta\). Generators for the latter representation can be found in \cite{ATLAS} for the case of \(G=\mathcal{A}_7, L_3(4)\) and in \cite{Hohn2014FiniteGO} for the case \(G=M_{10}\). 
 Note that in all cases, \(\widetilde{G}\) has exactly two irreducible \(10\)-dimensional representations dual to each other. We denote them by \(W_{10}\) and \(W_{10}'\)
 
 By the description of the group \(\widetilde{G}\), there is an induced representation $\widetilde{G}\rightarrow\GL(\bigwedge\nolimits^3 V_6)$,
 which we denote $W$, which descends to a faithful projective representation \(G\rightarrow\PGL(\bigwedge^3 V_6)\). A direct computation with characters shows the following lemma.
 
\begin{lem} 
The representation $W$ decomposes as the direct sum \(W = (A_1,\rho_1) \oplus (A_2,\rho_2) \) of the two irreducible 10-dimensional representations $(A_1,\rho_1)\cong W_{10}$ and $(A_2,\rho_2)\cong W'_{10}$ of the group $\widetilde{G}$. 
Moreover the underlying vector spaces $A_1,A_2\subset\bigwedge\nolimits^3 V_6$ of those representations are Lagrangian. 
\end{lem}
Observe that by construction, the two representations are dual to each other, and hence \(\delta(A_1)=A_2\). 
From \autoref{two_to_one} we deduce that the two Lagrangians have the same associated double EPW-cubes. For this reason we let \(A\) be any of the two Lagrangian spaces \(A_1\) or \(A_2\).
\begin{thm}\label{symmetric examples}
    Let \([A]\in\LG\) be the \(G\)-invariant Lagrangian space as above for \(G\) in \autoref{group_actions} and consider the associated double EPW-cube \(\dcub{2}\). We have that \([A]\not\in\Sigma\). 
    Moreover, the following holds: \begin{enumerate}

\item If \(G=L_3(4)\) then \([A]\in\Gamma\), in particular \(\dcub{2}\) is singular and \(G\leq \Bir^s_H(\dcub{2}^\epsilon)\), where \(\dcub{2}^\epsilon\) is a symplectic resolution. 

    \item If \(G=\mathcal{A}_7,M_{10},L_2(11)\) then \([A]\not\in\Gamma\), in particular \(\dcub{2}\) is smooth  and we have \(\Aut^s_H(\dcub{2})\cong G\).

    \end{enumerate}
    Moreover,
    \[\Aut_H(\dcub{2})\cong\left\{\begin{array}{cl} G\times \mathbb{Z}/4\mathbb{Z}& \text{ if $G=L_2(11)$}\\ G\times \mathbb{Z}/2\mathbb{Z}&\text{ if $G = \mathcal{A}_7,\, M_{10}$.}\end{array}\right .\]
\end{thm}
\begin{proof}
    First we prove that \([A]\not\in\Sigma\). The cases \(G=L_2(11)\) and \(G=\mathcal{A}_7\) were already treated in \cite{debarre2022gushel} and \cite{billi2022double} respectively. For the remaining cases we argue as in the proof of \cite[Proposition 2.2]{billi2022double} and perform the same computations using the code in \cite[Appendix A.2]{billi2022double} to get that \([A]\not\in\Sigma\). 
    Now we check when \([A]\in \Gamma\).
    \begin{enumerate}
        \item  If \([A]\not\in\Gamma\), then the manifold \(\dcub{2}\) is a smooth hyperkähler manifold and \(G\leq \Aut^s_H(\dcub{2})\). We deduce from \autoref{lattices_of_examples} that the group \(G=L_3(4)\) cannot have a regular symplectic action on a manifold of \(\KTcube\)--type and this implies that \([A]\in\Gamma\) for the \(L_3(4)\)-invariant Lagrangian space \(A\). The argument in the proof of \autoref{automorphisms_double_cover} shows that \(\PGL(V_6)_A\hookrightarrow \Aut_H(\dcub{2})\) even when \([A]\in\Gamma\setminus\Sigma\), from which follows that in this case we have an embedding \(G\hookrightarrow \Bir^s_H(\dcub{2}^\epsilon)\). 

    \item For the other groups, if \([A]\in\Gamma\), then \(\dcub{2}\) is singular and it admits a symplectic resolution \(\dcub{2}^\epsilon\) which is a projective hyperk\"ahler manifold of $\KTcube$--type (\autoref{epws}). Such a resolution is given by some \(\mathbb{P}^3\)--contractions. By \cite{hassett2010intersection}, to each contraction corresponds a divisor $D$ of numerical type $(-12,2)$ which is orthogonal to the class \(H\) given by the pullback of the polarization of \(\dcub{2}\). For the groups \(G = \mathcal{A}_7, L_2(11)\), there is a canonical \(G\)-invariant EPW-cube polarization, and its orthogonal in the algebraic part is exactly \((\bL_{\KTcube})_G\), which contains no vectors of type \((-12,2)\). Hence, for those groups, we can apply \autoref{lattices_of_examples} and deduce that \([A]\not\in\Gamma\).
We are left to deal with the case \(G=M_{10}\), for which we have two possible actions, one of which cannot be regular. We check that \(\cub{4}=\emptyset\) by a local computation in each affine chart, similarly to \cite[Proposition 2.2]{billi2022double} and using the local description of \cite[Lemma 2.7]{iliev2019epw}\footnote{we include the script used for such computations in \cite{database}}. This implies that \(\dcub{2}\) is smooth and \([A]\not\in\Gamma\) is \(G\)-invariant. Then by \autoref{automorphi_are_linear} and \autoref{automorphisms_double_cover} we get an inclusion \(G\leq\Aut^s_H(\dcub{2})\).  We obtain an isomorphism \(G\cong\Aut^s_H(\dcub{2})\) from the fact that the groups in \autoref{group_actions} are 
maximal for the inclusion (see \autoref{tab: lovely table 2}). 
\end{enumerate}
For the last part of the statement, it is immediate to see that for \(G=A_7, M_{10}\) the transcendental lattice admits only isometries of order at most 2 and hence \(\Aut_H(\dcub{2})\cong G\times \langle\iota\rangle\).

We now deal with the case \(G=L_2(11)\). From the description of \(A\) in \cite{debarre2022gushel}, there exists an involution \(\varphi\in\PGL(V_6)\) of determinant \(-1\) such that \(\delta(A)=f(A)\) where \(f:=\bigwedge^3 \varphi\). We also notice that by construction, \(\varphi\) commutes with all the elements of \(G\). 
According to \cite{debarre2020double}, the \(\mathcal{O}_{\cub{2}}\)-algebra structure on the sheaf \(\mathcal{O}_{\cub{2}}\oplus \mathcal{R}_2(-2)\) is given by the isomorphism \(\mathcal{R}_2(-2)\otimes \mathcal{R}_2(-2)\to \mathcal{O}_{\cub{2}}\) induced by multiplication for a section of \(\mathcal{O}_{\cub{2}}(4)\) and the isomorphism 
\[\det\left(V_6\wedge \bigwedge^2 \mathcal{U}\right)\cong\mathcal{O}_{\Gr}(-4),\] 
where \(\mathcal{U}\) denotes the tautological bundle of \(\Gr\) and 
\[V_6\wedge \bigwedge^2 \mathcal{U}=\Img\left(V_6\otimes \bigwedge^2\mathcal{U}\to \bigwedge^3 V_6\otimes\mathcal{O}_{\Gr}\right).\]
As a consequence, we have that the map \(g^*\colon \mathcal{O}_{\cub{2}}(-4)\to \mathcal{O}_{\cub{2}}(-4)\) is given by pullback by \(\delta\) composed with multiplication by \((-1)\) on the fibers. From this we see that composing the action induced by \(g\) with multiplication by \(i\) on \(\mathcal{R}_2(-2)\) gives an automorphism of the $\mathcal{O}_{\cub{2}}$-algebra $\mathcal{O}_{\cub{2}}\oplus \mathcal{R}_2(-2)$, this induces an automorphism $\widetilde{\delta}\in \Aut_H(\dcub{2})$ whose square coincides with $\iota$, the covering involution. Since $\iota$ is nonsymplectic, we have that $\widetilde{\delta}$ is purely nonsymplectic of order 4, and we have a split extension
    \[1\to\PGL(V_6)_A \to \Aut_H(\dcub{2})\to \langle \widetilde{\delta}\rangle\to 1.\]
    Note that the previous extension gives rise to a decomposition of \(\Aut_H(\dcub{2})\) as a direct product of \(\PGL(V_6)_A\) and \(\langle \widetilde{\delta}\rangle\) if and only if $(\bigwedge^3\varphi)\circ\delta$ is central in $\PGL(V_6)_A$. The latter is equivalent to $\varphi$ commuting with every element in $\PGL(V_6)_A$, this is the case and hence we can conclude that \(\Aut_H(\dcub{2})\cong G\times \langle\widetilde{\delta}\rangle\) \footnote{see the notebook "ExtraAut" of \cite{database} for a computational lattice-theoretic evidence of the existence of such a purely nonsymplectic automorphism}.

\end{proof}

\begin{rmk}
    In \cite[Theorem 0.3, Figure 3]{muller} the second named author showed the existence of a hyperk\"ahler manifold $X$ of deformation type $\KTcube$ and of Picard rank 1, equipped with a nonstable purely nonsymplectic automorphism $g$ of order 4 acting trivially on $\NS(X)$. In such a situation, it is proved that $\NS(X)\simeq [4]$ and $\T(X)\simeq \bU^{\oplus 2}\oplus \bE_8^{\oplus 2}\oplus[-2]^{\oplus2}$. Note that the previous numerical data determine a 9-dimensional family of manifolds of K3\(^{[3]}\)-type
    whose general element is equipped with a purely nonsymplectic automorphism of order 4. The $L_2(11)$-invariant Lagrangian from \autoref{symmetric examples} gives an explicit geometric realization for this existence statement, where the order 4 non-symplectic automorphism squares to the associated covering involution.. 
\end{rmk}

\appendix
\section{Table of stably saturated stable symplectic finite subgroups for manifolds of \texorpdfstring{\(\KTcube\)}{K3[3]}--type}\label{sec appendix}
The lattice-theoretic counterpart of \autoref{class_symplectic_groups} is presented in \autoref{tab: lovely table 2}. Each entry corresponds to an $\bO^+(\bL_{\KTcube})$-conjugacy class $\mathcal{G}$ of saturated finite subgroups of $\widetilde{O}(\bL_{\KTcube})$. For each such conjugacy class $\mathcal{G}$, we denote by $G\leq \widetilde{O}(\bL_{\KTcube})$ a representative and by $C := (\bL_{\KTcube})_G$ the associated coinvariant sublattice. Then the corresponding entry in the table gives:
\begin{itemize}
    \item the label of the associated lattice $C$ from H\"ohn--Mason database \cite[Table 2]{HMLeech}. In case of several $\bO(\bL_{\KTcube})$-orbits of primitive sublattices (see \autoref{class_symplectic_groups}), we add letters to distinguish each orbit;
    \item a description of the group $O^\#(C)\cong G$, or its Id in the Small Group Library \cite{SGL}, or its order;
    \item the genus of the invariant sublattice $\bL_{\textnormal{K3}^{[3]}}^G$ following Conway--Sloane convention \cite[Chapter 15]{splg}.
\end{itemize} 
Moreover, in each case, we also give the number of classes of numerical type \((-12,2)\) and type \((-36,4)\) in $C$ (\autoref{easy divisibility}).

{
\small\centering\setlength{\tabcolsep}{1pt}
\renewcommand\arraystretch{1.5}
\rowcolors{1}{lightgray!40!white}{white}
\begin{longtable}{cccc|cccc}

\rowcolor{white}\caption{Stably saturated stable symplectic finite subgroups for the deformation type $\textnormal{K3}^{[3]}$}
\label{tab: lovely table 2}\vspace*{0.3cm}\\
 Id& $O^\#(C)$& $\#(-12, -36)$& $g(\bL_{\textnormal{K3}^{[3]}}^G)$&Id& $ O^\#(C)$& $\#(-12, -36)$& $g(\bL_{\textnormal{K3}^{[3]}}^G)$\\
\hline
\endfirsthead

\rowcolor{white}\caption[]{Stably saturated symplectic finite stable groups for the deformation type $\textnormal{K3}^{[3]}$ (continued)}\vspace*{0.3cm}\\
 Id& $O^\#(C)$& $\#(-12, -36)$& $g(\bL_{\textnormal{K3}^{[3]}}^G)$&Id& $ O^\#(C)$& $\#(-12, -36)$& $g(\bL_{\textnormal{K3}^{[3]}}^G)$\\
\hline
\endhead

1 & $C_1$ & $(0, 0)$& $ \II_{(3, 20)}4^{1}_{7} $ & 71b & $[384, 20164]$ & $(312, 64)$&  $ \II_{(3, 1)}2^{-2}4^{-2}_{6} $  \\

2 & $C_2$&$(0, 0)$& $ \II_{(3, 12)}2^{-8}4^{-1}_{3} $ & \multirow{1}{*}{72a} & \multirow{1}{*}{${A}_6$} &$(0, 0)$  &  $ \II_{(3, 1)}4^{-2}_{2}3^{2}5^{1} $  \\                                                                                   
\multirow{1}{*}{3a} &\multirow{1}{*}{$C_2^2$} &$(0, 0)$ & $ \II_{(3, 8)}2^{-6}4^{-3}_{7} $ &   \multirow{1}{*}{72b} & \multirow{1}{*}{${A}_6$}&$(380, 0)$ &  $ \II_{(3, 1)}2^{-2}_{6}3^{2}5^{1} $  \\                                                                                    
 \multirow{1}{*}{3b} &\multirow{1}{*}{$C_2^2$}&$(64, 0)$& $ \II_{(3, 8)}2^{-4}4^{-3}_{7} $ &   \multirow{1}{*}{73a} & \multirow{1}{*}{${A}_{4,4}$}&$(0, 0)$  &  $ \II_{(3, 1)}2^{2}4^{1}_{7}8^{1}_{7}3^{2} $  \\                                                                            
4 &$C_3$&$(0, 0)$ & $ \II_{(3, 8)}4^{1}_{7}3^{6} $ &  \multirow{1}{*}{73b} & \multirow{1}{*}{${A}_{4,4}$}&$(304, 0)$ &  $ \II_{(3, 1)}4^{1}_{7}8^{1}_{7}3^{2} $  \\                                                                                   
5 &$C_2$ &$(32, 0)$ & $ \II_{(3, 8)}2^{10}_{4}4^{1}_{7} $  & \multirow{1}{*}{74a} &\multirow{1}{*}{$H_{192}$} &$(0, 0)$ &  $ \II_{(3, 1)}4^{-3}_{5}8^{1}_{7}3^{1} $ \\                                                                             
\multirow{1}{*}{6a} & \multirow{1}{*}{$C_2^3$}&$(0, 0)$ & $ \II_{(3, 6)}2^{6}4^{3}_{5} $ &   \multirow{1}{*}{74b} &\multirow{1}{*}{$H_{192}$}&$(286, 0)$ &  $ \II_{(3, 1)}2^{2}_{2}4^{1}_{1}8^{-1}_{5}3^{1} $  \\                                                                         
\multirow{1}{*}{6b} & \multirow{1}{*}{$C_2^3$}&$(96, 0)$ & $ \II_{(3, 6)}2^{-4}4^{-3}_{1} $ &   \multirow{1}{*}{75a} & \multirow{1}{*}{$[192, 1538]$}&$(280, 0)$ &  $ \II_{(3, 1)}2^{2}4^{-1}_{3}8^{1}_{1}3^{1} $ \\                                                                           
 \multirow{1}{*}{6c} & \multirow{1}{*}{$C_2^3$}&$(112, 0)$& $ \II_{(3, 6)}2^{-8}_{6}4^{-1}_{3} $ &   \multirow{1}{*}{75b} & \multirow{1}{*}{$[192, 1538]$}&$(240, 0)$&  $ \II_{(3, 1)}2^{2}4^{-1}_{3}8^{1}_{1}3^{1} $ \\                                                                       
7 & ${S}_3$&$(0, 0)$ & $ \II_{(3, 6)}2^{2}4^{-1}_{3}3^{-5} $ &\multirow{1}{*}{76a} & \multirow{1}{*}{$T_{192}$}&$(0, 0)$ &  $ \II_{(3, 1)}4^{-4}_{0}3^{-1} $  \\                                                                                   
8 & $C_2^2$ &$(48, 0)$ & $ \II_{(3, 6)}2^{6}4^{3}_{5} $ &   \multirow{1}{*}{76b} & \multirow{1}{*}{$T_{192}$}&$(370, 0)$ &  $ \II_{(3, 1)}2^{-2}_{4}4^{2}_{0}3^{-1} $  \\                                                                                 
\multirow{1}{*}{9a} & \multirow{1}{*}{$C_4$}&$(0, 0)$ & $ \II_{(3, 6)}2^{-2}_{2}4^{5}_{7} $ &   \multirow{1}{*}{76c} & \multirow{1}{*}{$T_{192}$}&$(360, 0)$&  $ \II_{(3, 1)}2^{2}4^{-2}3^{-1} $  \\                                                                                    
\multirow{1}{*}{9b} & \multirow{1}{*}{$C_4$}&$(104, 0)$  & $ \II_{(3, 6)}4^{-5}_{5} $ &   \multirow{1}{*}{76d} & \multirow{1}{*}{$T_{192}$}&$(360, 64)$ &  $ \II_{(3, 1)}4^{-2}3^{-1} $  \\                                                                                                  
\multirow{1}{*}{10a} &\multirow{1}{*}{$C_2^4$} &$(0, 0)$ & $ \II_{(3, 5)}2^{6}4^{1}_{7}8^{1}_{7} $ &   \multirow{1}{*}{77a} & \multirow{1}{*}{$L_2(7)$}&$(0, 0)$&  $ \II_{(3, 1)}4^{2}_{6}7^{2} $  \\                                                                                  
\multirow{1}{*}{10b} &\multirow{1}{*}{$C_2^4$}&$(160, 0)$ & $ \II_{(3, 5)}2^{-4}4^{-1}_{3}8^{1}_{7} $ & \multirow{1}{*}{77b} & \multirow{1}{*}{$L_2(7)$}&$(364, 0)$ &  $ \II_{(3, 1)}2^{2}_{6}7^{2} $ \\                                                                                
11 & $C_2^4$&$(112, 0)$ & $ \II_{(3, 5)}2^{6}4^{2}_{6} $  &\multirow{1}{*}{79a} & \multirow{1}{*}{$[128, 1758]$} &$(192, 0)$ &  $ \II_{(3, 1)}4^{-3}_{3}8^{-1}_{3} $  \\                                                                                     
12 & $C_2^3$&$(80, 0)$ & $ \II_{(3, 5)}2^{6}4^{-1}_{5}8^{-1}_{5} $ &   \multirow{1}{*}{79b} & \multirow{1}{*}{$[128, 1758]$} &$(262, 0)$  &  $ \II_{(3, 1)}4^{-3}_{3}8^{-1}_{3} $  \\                                                                          
\multirow{1}{*}{13a} & \multirow{1}{*}{$D_4$}&$(0, 0)$ & $ \II_{(3, 5)}4^{6}_{6} $  & \multirow{1}{*}{81a} & \multirow{1}{*}{$[128, 1755]$}&$(224, 0)$  &  $ \II_{(3, 1)}4^{3}_{1}8^{1}_{1} $  \\                                                                                            
\multirow{1}{*}{13b} & \multirow{1}{*}{$D_4$}&$(124, 0)$ & $ \II_{(3, 5)}2^{2}_{0}4^{4}_{6} $ &   \multirow{1}{*}{81b} & \multirow{1}{*}{$[128, 1755]$}&$(260, 0)$ &  $ \II_{(3, 1)}4^{-3}_{3}8^{-1}_{3} $ \\                                                                                 
14 & $\#512$&$(256, 0)$ & $ \II_{(3, 4)}2^{-6}4^{-1}_{3} $  &82 & ${S}_5$&$(0, 0)$ &  $ \II_{(3, 1)}4^{-2}_{4}3^{-1}5^{-2} $ \\                                                                                 
15 & ${A}_{3,3}$&$(0, 0)$ & $ \II_{(3, 4)}4^{1}_{7}3^{4}9^{1} $  & 83 & $C_2^2\times{S}_4$&$(208, 0)$  &  $ \II_{(3, 1)}2^{2}4^{2}_{6}3^{2} $  \\                                                                                 
\multirow{1}{*}{16a} &\multirow{1}{*}{$C_2^4$} &$(128, 0)$& $ \II_{(3, 4)}2^{4}4^{3}_{7} $ &  \multirow{1}{*}{84a} & \multirow{1}{*}{$M_9$}&$(0, 0)$ &  $ \II_{(3, 1)}2^{3}_{7}4^{-1}_{5}3^{1}9^{1} $  \\                                                                            
 \multirow{1}{*}{16b} &\multirow{1}{*}{$C_2^4$}&$(144, 0)$  & $ \II_{(3, 4)}2^{4}4^{3}_{7} $  &   \multirow{1}{*}{84b} & \multirow{1}{*}{$M_9$}&$(344, 0)$&  $ \II_{(3, 1)}2^{1}_{7}4^{-1}_{5}3^{1}9^{1} $ \\                                                                            
\multirow{1}{*}{17a} & \multirow{1}{*}{$\Gamma_2a_1$}&$(0, 0)$ & $ \II_{(3, 4)}2^{2}4^{5}_{7} $  & 85& $N_{72}$&$(0, 0)$ &  $ \II_{(3, 1)}4^{2}_{6}3^{2}9^{1} $  \\                                                                                      
 \multirow{1}{*}{17b} & \multirow{1}{*}{$\Gamma_2a_1$}&$(144, 0)$  & $ \II_{(3, 4)}4^{5}_{7} $  & 86 & $D_4^2$&$(188, 32)$ &  $ \II_{(3, 1)}4^{4}_{2} $ \\                                                                                                     
 \multirow{1}{*}{17c} & \multirow{1}{*}{$\Gamma_2a_1$}&$(128, 0)$ & $ \II_{(3, 4)}4^{5}_{7} $  & 87 & $T_{48}$&$(0, 0)$ &  $ \II_{(3, 1)}2^{1}_{7}4^{1}_{1}8^{-2}3^{1} $  \\                                                                                 
 \multirow{1}{*}{17d} & \multirow{1}{*}{$\Gamma_2a_1$}&$(148, 0)$ & $ \II_{(3, 4)}2^{4}_{2}4^{3}_{5} $ &  88 & $C_2\times{S}_4$&$(180, 0)$ &  $ \II_{(3, 1)}4^{-4}_{0}3^{-1} $  \\                                                                                     
18 & $D_6$ &$(0, 0)$ & $ \II_{(3, 4)}2^{4}4^{1}_{7}3^{4} $  & \multirow{1}{*}{89a} & \multirow{1}{*}{$[32,43]$}&$(208, 0)$  &  $ \II_{(3, 1)}2^{-1}_{3}4^{-1}_{3}8^{2} $ \\                                                                           
\multirow{1}{*}{19a} & \multirow{1}{*}{${A}_4$} &$(0, 0)$& $ \II_{(3, 4)}2^{-2}4^{-3}_{7}3^{2} $  &  \multirow{1}{*}{89b} & \multirow{1}{*}{$[32,43]$}&$(244, 0)$ &  $ \II_{(3, 1)}2^{-1}_{3}4^{1}_{7}8^{-2} $  \\                                                                         
\multirow{1}{*}{19b} & \multirow{1}{*}{${A}_4$}&$(172, 0)$ & $ \II_{(3, 4)}4^{3}_{3}3^{2} $  & 90 & ${S}_4$&$(188, 0)$ &  $ \II_{(3, 1)}2^{-2}4^{-1}_{3}8^{1}_{7}3^{2} $  \\                                                                           
20 & $D_5$ &$(0, 0)$ & $ \II_{(3, 4)}4^{1}_{7}5^{4} $  & 91 & ${S}_4$&$(156, 24)$&  $ \II_{(3, 1)}4^{4}_{2} $ \\                                                                                                
21 & $D_4$&$(72, 0)$ & $ \II_{(3, 4)}2^{-2}4^{-5}_{3} $  & 92 &${S}_4$ &$(270, 0)$ &  $ \II_{(3, 1)}2^{-2}_{6}4^{-1}_{5}8^{-1}_{5}3^{-1} $\\                                                                   
23 & $C_2^2$ &$(72, 0)$ & $ \II_{(3, 4)}2^{-2}_{6}4^{-5}_{5} $ & 94 & $C_2\times Q_8$&$(316, 0)$&  $ \II_{(3, 1)}2^{-2}_{4}4^{1}_{7}16^{-1}_{3} $  \\                                                                     
\multirow{1}{*}{24a} & \multirow{1}{*}{$C_2^2\rtimes {A}_4$}&$(0, 0)$ & $ \II_{(3, 3)}2^{4}4^{-1}_{3}8^{1}_{7}3^{1} $  & 95 &$C_2\times D_4$ &$(196, 8)$&  $ \II_{(3, 1)}2^{-2}_{4}8^{-2}_{6} $  \\                                                                      
 \multirow{1}{*}{24b} & \multirow{1}{*}{$C_2^2\rtimes {A}_4$}&$(232, 0)$ & $ \II_{(3, 3)}2^{2}4^{1}_{1}8^{-1}_{5}3^{1} $  & 97& $D_6$ &$(230, 0)$ &  $ \II_{(3, 1)}2^{2}_{2}4^{2}_{0}3^{-2} $  \\                                                                  
\multirow{1}{*}{25a} & \multirow{1}{*}{$C_2^5$}&$(208, 0)$ & $ \II_{(3, 3)}2^{4}4^{1}_{1}8^{1}_{7} $  & 99 & $\#245760$&$(640, 0)$ &  $ \II_{(3, 0)}4^{3}_{3} $  \\                                                                                       
\multirow{1}{*}{25b} & \multirow{1}{*}{$C_2^5$}&$(192, 0)$& $ \II_{(3, 3)}2^{4}4^{1}_{1}8^{1}_{7} $  & 100 & $\#30720$&$(576, 0)$&  $ \II_{(3, 0)}2^{2}4^{-1}_{3}5^{-1} $  \\                                                                          
\multirow{1}{*}{26a} & \multirow{1}{*}{$\Gamma_4a_1$}&$(0, 0)$ & $ \II_{(3, 3)}2^{2}4^{3}_{1}8^{1}_{7} $  & \multirow{1}{*}{102a} & \multirow{1}{*}{$L_3(4)$}&$(0, 0)$  &  $ \II_{(3, 0)}2^{2}4^{-1}_{3}3^{1}7^{1} $  \\                                                                      
 \multirow{1}{*}{26b} & \multirow{1}{*}{$\Gamma_4a_1$}&$(208, 0)$&  $ \II_{(3, 3)}4^{3}_{1}8^{1}_{7} $ &\multirow{1}{*}{102b} & \multirow{1}{*}{$L_3(4)$}&$(0, 0)$  &  $ \II_{(3, 0)}2^{-2}4^{1}_{7}3^{1}7^{1} $  \\                                                                           
 \multirow{1}{*}{26c} & \multirow{1}{*}{$\Gamma_4a_1$}&$(214, 0)$  & $ \II_{(3, 3)}2^{-4}_{2}4^{1}_{1}8^{-1}_{5} $  &  \multirow{1}{*}{102c} & \multirow{1}{*}{$L_3(4)$}&$(560, 0)$ &  $ \II_{(3, 0)}4^{-1}_{3}3^{1}7^{1} $ \\                                                                     
\multirow{1}{*}{27a} & \multirow{1}{*}{$C_2\times \Gamma_2a_1$}&$(144, 0)$ & $ \II_{(3, 3)}2^{2}4^{4}_{0} $  &  \multirow{1}{*}{102d} & \multirow{1}{*}{$L_3(4)$}&$(560, 0)$ &  $ \II_{(3, 0)}4^{-1}_{3}3^{1}7^{1} $  \\                                                                                    
 \multirow{1}{*}{27b} & \multirow{1}{*}{$C_2\times \Gamma_2a_1$}&$(168, 0)$& $ \II_{(3, 3)}2^{-2}4^{-4}_{4} $  & 103 & $\#12288$&$(480, 0)$  &  $ \II_{(3, 0)}4^{-2}_{4}8^{-1}_{3} $  \\                                                                                  
\multirow{1}{*}{28a} & \multirow{1}{*}{$2_+^{1+4}$}&$(0, 0)$  & $ \II_{(3, 3)}4^{6}_{0} $  & 104 & $\#9216$ &$(448, 0)$ &  $ \II_{(3, 0)}2^{2}4^{1}_{7}3^{2} $  \\                                                                                          
 \multirow{1}{*}{28b} & \multirow{1}{*}{$2_+^{1+4}$}&$(172, 0)$ & $ \II_{(3, 3)}2^{-2}_{4}4^{-4}_{0} $  & 105 & $\#6144$&$(416, 0)$  &  $ \II_{(3, 0)}4^{-3}_{5}3^{1} $  \\                                                                                   
  \multirow{1}{*}{28c} & \multirow{1}{*}{$2_+^{1+4}$}&$(168, 0)$ & $ \II_{(3, 3)}2^{2}4^{4} $  & \multirow{1}{*}{106a} & \multirow{1}{*}{$C_2^4\rtimes A_6$}&$(0, 0)$ &  $ \II_{(3, 0)}4^{2}_{0}8^{-1}_{5}3^{-1} $  \\                                                                                   
  \multirow{1}{*}{28d} & \multirow{1}{*}{$2_+^{1+4}$}&$(168, 16)$ & $ \II_{(3, 3)}4^{4} $  &  \multirow{1}{*}{106b} & \multirow{1}{*}{$C_2^4\rtimes A_6$}&$(520, 0)$ &  $ \II_{(3, 0)}2^{2}8^{-1}_{5}3^{-1} $  \\                                                                                            
\multirow{1}{*}{29a} & \multirow{1}{*}{${S}_4$}&$(0, 0)$ & $ \II_{(3, 3)}4^{4}_{4}3^{2} $  &  \multirow{1}{*}{106c} & \multirow{1}{*}{$C_2^4\rtimes A_6$}&$(530, 0)$  &  $ \II_{(3, 0)}2^{2}_{2}8^{-1}_{3}3^{-1} $  \\                                                                               
 \multirow{1}{*}{29b} & \multirow{1}{*}{${S}_4$}&$(196, 0)$ & $ \II_{(3, 3)}2^{-2}_{4}4^{-2}_{4}3^{2} $  &  \multirow{1}{*}{106d} & \multirow{1}{*}{$C_2^4\rtimes A_6$}&$(520, 80)$  &  $ \II_{(3, 0)}8^{-1}_{5}3^{-1} $  \\                                                                             
30 & $D_4$&$(104, 0)$  & $ \II_{(3, 3)}2^{2}_{0}4^{3}_{1}8^{1}_{7} $  & \multirow{1}{*}{108a} & \multirow{1}{*}{${A}_7$} &$(0, 0)$ &  $ \II_{(3, 0)}4^{1}_{7}3^{-1}5^{1}7^{-1} $  \\                                                                 
\multirow{1}{*}{31a} & \multirow{1}{*}{$Q_8$}&$(0, 0)$ & $ \II_{(3, 3)}2^{-3}_{1}4^{1}_{7}8^{-2} $  &   \multirow{1}{*}{108b} & \multirow{1}{*}{${A}_7$}&$(0, 0)$ &  $ \II_{(3, 0)}4^{1}_{7}3^{-1}5^{1}7^{-1} $  \\                                                                   
 \multirow{1}{*}{31b} & \multirow{1}{*}{$Q_8$}&$(200, 0)$  & $ \II_{(3, 3)}2^{1}_{7}4^{-1}_{5}8^{-2} $ &  \multirow{1}{*}{108c} & \multirow{1}{*}{${A}_7$}&$(0, 0)$ &  $ \II_{(3, 0)}4^{1}_{7}3^{-1}5^{1}7^{-1} $ \\                                                                   
32 & $C_2^3$ &$(84, 0)$ & $ \II_{(3, 3)}4^{6}_{0} $ &  \multirow{1}{*}{110a} &\multirow{1}{*}{$[1920, 240993]$} &$(0, 0)$ &  $ \II_{(3, 0)}4^{-2}_{2}8^{1}_{1}5^{-1} $  \\                                                                                    
33 & $\#1536$&$(352, 0)$ & $ \II_{(3, 2)}2^{4}4^{-1}_{3}3^{1} $ &   \multirow{1}{*}{110b} &\multirow{1}{*}{$[1920, 240993]$} &$(0, 0)$  &  $ \II_{(3, 0)}4^{-2}_{2}8^{1}_{1}5^{-1} $  \\                                                                         
34 & $\#1024$&$(320, 0)$ & $ \II_{(3, 2)}2^{2}4^{3}_{1} $ &  \multirow{1}{*}{110c} &\multirow{1}{*}{$[1920, 240993]$} &$(0, 0)$ &  $ \II_{(3, 0)}4^{-2}_{2}8^{1}_{1}5^{-1} $  \\                                                                               
\multirow{1}{*}{36a} &\multirow{1}{*}{$[192, 1541]$} &$(304, 0)$& $ \II_{(3, 2)}2^{2}4^{3}_{1} $ & \multirow{1}{*}{110d} &\multirow{1}{*}{$[1920, 240993]$}   &$(430, 0)$ &  $ \II_{(3, 0)}2^{-2}_{6}8^{1}_{1}5^{-1} $  \\                                                                               
 \multirow{1}{*}{36b} &\multirow{1}{*}{$[192, 1541]$}&$(288, 64)$  & $ \II_{(3, 2)}2^{4}4^{1}_{1} $ & \multirow{1}{*}{110e} &\multirow{1}{*}{$[1920, 240993]$}  &$(430, 0)$ &  $ \II_{(3, 0)}2^{-2}_{6}8^{1}_{1}5^{-1} $  \\                                                                               
\multirow{1}{*}{37a} & \multirow{1}{*}{$C_4^2\rtimes {A}_4$}&$(0, 0)$ & $ \II_{(3, 2)}2^{2}4^{-1}_{3}8^{-2}_{2} $  & \multirow{1}{*}{111a} &\multirow{1}{*}{$[1344, 11686]$} &$(0, 0)$ &  $ \II_{(3, 0)}4^{3}_{5}7^{-1} $ \\                                                                              
 \multirow{1}{*}{37b} & \multirow{1}{*}{$C_4^2\rtimes {A}_4$}&$(304, 0)$ & $ \II_{(3, 2)}4^{1}_{1}8^{2}_{0} $ &   \multirow{1}{*}{111b} &\multirow{1}{*}{$[1344, 11686]$}&$(490, 0)$ &  $ \II_{(3, 0)}2^{-2}_{6}4^{-1}_{3}7^{-1} $  \\                                                                          
\multirow{1}{*}{38a} & \multirow{1}{*}{$C_2^2\rtimes{S}_4$}&$(0, 0)$ & $ \II_{(3, 2)}2^{2}4^{-2}_{4}8^{1}_{7}3^{1} $  & 112 & $[1152, 155478]$&$(0, 0)$&  $ \II_{(3, 0)}4^{1}_{7}8^{-2}_{2}3^{1} $ \\                                                                 
 \multirow{1}{*}{38b} & \multirow{1}{*}{$C_2^2\rtimes{S}_4$}&$(256, 0)$  & $ \II_{(3, 2)}4^{-2}_{4}8^{1}_{7}3^{1} $  & \multirow{1}{*}{113a} &\multirow{1}{*}{$[1152, 157862]$} &$(376, 0)$  &  $ \II_{(3, 0)}4^{-3}_{5}3^{1} $  \\                                                                               
 \multirow{1}{*}{38c} & \multirow{1}{*}{$C_2^2\rtimes{S}_4$}&$(262, 0)$  & $ \II_{(3, 2)}2^{4}_{2}8^{-1}_{5}3^{1} $ &   \multirow{1}{*}{113b} &\multirow{1}{*}{$[1152, 157862]$}&$(360, 64)$ &  $ \II_{(3, 0)}2^{2}4^{-1}_{5}3^{1} $  \\                                                                          
39 & ${A}_{4,3}$&$(0, 0)$& $ \II_{(3, 2)}4^{-3}_{7}3^{3} $  & \multirow{1}{*}{115a} &\multirow{1}{*}{$[768, 1090134]$} &$(288, 0)$ &  $ \II_{(3, 0)}4^{-1}_{3}8^{-2}_{4} $ \\                                                                                   
\multirow{1}{*}{40a} &\multirow{1}{*}{$C_2\times 2_+^{1+4}$} &$(160, 0)$ & $ \II_{(3, 2)}4^{5}_{1} $ &   \multirow{1}{*}{115b} &\multirow{1}{*}{$[768, 1090134]$}&$(358, 0)$ &  $ \II_{(3, 0)}4^{-1}_{3}8^{-2}_{4} $  \\                                                                                         
 \multirow{1}{*}{40b} &\multirow{1}{*}{$C_2\times 2_+^{1+4}$}&$(196, 0)$& $ \II_{(3, 2)}4^{5}_{1} $  & 117 & $[768, 1090070]$ &$(336, 64)$  &  $ \II_{(3, 0)}4^{3}_{3} $  \\                                                                                                    
\multirow{1}{*}{42a} & \multirow{1}{*}{$[64,202]$}&$(208, 0)$& $ \II_{(3, 2)}2^{2}4^{-2}_{2}8^{-1}_{3} $  & \multirow{1}{*}{118a} &\multirow{1}{*}{${S}_6$} &$(0, 0)$&  $ \II_{(3, 0)}2^{2}4^{-1}_{3}3^{2}5^{1} $  \\                                                                    
 \multirow{1}{*}{42b} & \multirow{1}{*}{$[64,202]$}&$(232, 0)$ & $ \II_{(3, 2)}2^{2}4^{-2}_{2}8^{-1}_{3} $  &  \multirow{1}{*}{118b} &\multirow{1}{*}{${S}_6$} &$(0, 0)$ &  $ \II_{(3, 0)}2^{2}4^{-1}_{3}3^{2}5^{1} $ \\                                                                    
\multirow{1}{*}{43a} & \multirow{1}{*}{$\Gamma_{25}a_1$}&$(0, 0)$ & $ \II_{(3, 2)}4^{4}_{0}8^{1}_{1} $ &   \multirow{1}{*}{118c} &\multirow{1}{*}{${S}_6$} &$(360, 0)$ &  $ \II_{(3, 0)}4^{-1}_{3}3^{2}5^{1} $  \\                                                                                
 \multirow{1}{*}{43b} & \multirow{1}{*}{$\Gamma_{25}a_1$}&$(232, 0)$& $ \II_{(3, 2)}2^{2}4^{2}8^{1}_{1} $ &   \multirow{1}{*}{118d} &\multirow{1}{*}{${S}_6$} &$(360, 0)$ &  $ \II_{(3, 0)}4^{-1}_{3}3^{2}5^{1} $  \\                                                                               
 \multirow{1}{*}{43c} & \multirow{1}{*}{$\Gamma_{25}a_1$}&$(236, 0)$& $ \II_{(3, 2)}2^{2}_{2}4^{2}8^{1}_{7} $ &  \multirow{1}{*}{119a} & \multirow{1}{*}{$M_{10}$} &$(0, 0)$ &  $ \II_{(3, 0)}2^{1}_{1}4^{2}_{0}3^{1}5^{1} $  \\                                                                   
 \multirow{1}{*}{43d} & \multirow{1}{*}{$\Gamma_{25}a_1$}&$(238, 0)$  & $ \II_{(3, 2)}2^{-2}_{2}4^{-2}_{4}8^{1}_{7} $ &   \multirow{1}{*}{119b} & \multirow{1}{*}{$M_{10}$} &$(0, 0)$  &  $ \II_{(3, 0)}2^{1}_{7}4^{2}_{2}3^{1}5^{1} $ \\                                                             
 \multirow{1}{*}{43e} & \multirow{1}{*}{$\Gamma_{25}a_1$}&$(232, 32)$ & $ \II_{(3, 2)}4^{2}8^{1}_{1} $  &  \multirow{1}{*}{119c} & \multirow{1}{*}{$M_{10}$} &$(0, 0)$ &  $ \II_{(3, 0)}2^{1}_{1}4^{2}_{0}3^{1}5^{1} $  \\                                                                            
\multirow{1}{*}{44a} & \multirow{1}{*}{${A}_5$}&$(0, 0)$ & $ \II_{(3, 2)}2^{-2}4^{1}_{7}3^{-1}5^{-2} $ &   \multirow{1}{*}{119d} & \multirow{1}{*}{$M_{10}$} &$(470, 0)$ &  $ \II_{(3, 0)}2^{3}_{1}3^{1}5^{1} $  \\                                                                        
 \multirow{1}{*}{44b} & \multirow{1}{*}{${A}_5$}&$(280, 0)$  & $ \II_{(3, 2)}4^{-1}_{3}3^{-1}5^{-2} $ &   \multirow{1}{*}{119e} & \multirow{1}{*}{$M_{10}$} &$(470, 0)$  &  $ \II_{(3, 0)}2^{3}_{1}3^{1}5^{1} $ \\                                                                             
\multirow{1}{*}{45a} & \multirow{1}{*}{$C_2\times {S}_4$}&$(0, 0)$ & $ \II_{(3, 2)}2^{2}4^{3}_{5}3^{2} $  &  \multirow{1}{*}{119f} & \multirow{1}{*}{$M_{10}$} &$(470, 72)$ &  $ \II_{(3, 0)}2^{1}_{1}3^{1}5^{1} $\\                                                                                
\multirow{1}{*}{45b} & \multirow{1}{*}{$C_2\times {S}_4$}&$(192, 0)$  & $ \II_{(3, 2)}4^{3}_{5}3^{2} $  & \multirow{1}{*}{120a} &\multirow{1}{*}{$L_2(11)$} &$(0, 0)$&  $ \II_{(3, 0)}4^{1}_{7}11^{2} $  \\                                                                                         
 \multirow{1}{*}{45c} & \multirow{1}{*}{$C_2\times {S}_4$}&$(220, 0)$& $ \II_{(3, 2)}2^{4}_{4}4^{1}_{1}3^{2} $ &   \multirow{1}{*}{120b} &\multirow{1}{*}{$L_2(11)$}&$(0, 0)$&  $ \II_{(3, 0)}4^{1}_{7}11^{2} $  \\                                                                                
\multirow{1}{*}{46a} & \multirow{1}{*}{$[36,9]$}&$(0, 0)$ & $ \II_{(3, 2)}2^{-2}_{6}4^{-1}_{3}3^{2}9^{1} $ &  \multirow{1}{*}{121a} &\multirow{1}{*}{$[576, 8654]$} &$(0, 0)$ &  $ \II_{(3, 0)}4^{2}_{0}8^{1}_{7}3^{2} $  \\                                                                 
 \multirow{1}{*}{46b} & \multirow{1}{*}{$[36,9]$}&$(272, 0)$ & $ \II_{(3, 2)}4^{-1}_{5}3^{2}9^{1} $ &   \multirow{1}{*}{121b} &\multirow{1}{*}{$[576, 8654]$}&$(0, 0)$&  $ \II_{(3, 0)}4^{2}_{0}8^{1}_{7}3^{2} $  \\                                                                           
47 & ${S}_{3}^2$&$(0, 0)$  & $ \II_{(3, 2)}2^{2}4^{-1}_{3}3^{-3}9^{1} $ &   \multirow{1}{*}{121c} &\multirow{1}{*}{$[576, 8654]$}&$(334, 0)$ &  $ \II_{(3, 0)}2^{-2}_{2}8^{-1}_{5}3^{2} $  \\                                                                   
48 & $\Gamma_5a_2$&$(168, 0)$ & $ \II_{(3, 2)}4^{5}_{1} $ &  122 &$[500, 23]$ &$(0, 0)$ &  $ \II_{(3, 0)}4^{1}_{7}5^{3} $  \\                                                                                               
\multirow{1}{*}{49a} & \multirow{1}{*}{$[32,27]$}&$(144, 0)$ & $ \II_{(3, 2)}4^{5}_{1} $ &  124 & $[384, 18134]$&$(456, 0)$ &  $ \II_{(3, 0)}2^{1}_{1}4^{-1}_{5}16^{-1}_{5} $  \\                                 

 \multirow{1}{*}{49b} & \multirow{1}{*}{$[32,27]$}&$(180, 0)$ & $ \II_{(3, 2)}4^{5}_{1} $ &  \multirow{1}{*}{125a} & \multirow{1}{*}{$[384,17948]$} &$(256, 0)$  &  $ \II_{(3, 0)}4^{-2}_{4}8^{1}_{1}3^{1} $  \\                                                                                     
 \multirow{1}{*}{49c} & \multirow{1}{*}{$[32,27]$}&$(168, 32)$ & $ \II_{(3, 2)}2^{-2}4^{-3}_{5} $ &    \multirow{1}{*}{125b} & \multirow{1}{*}{$[384,17948]$}&$(256, 0)$ &  $ \II_{(3, 0)}4^{-2}_{4}8^{1}_{1}3^{1} $  \\                                                                              
50 & $C_2^2\times{S}_3$&$(112, 0)$ & $ \II_{(3, 2)}2^{4}4^{-1}_{3}3^{-3} $  & \multirow{1}{*}{126a} &\multirow{1}{*}{$[384, 20089]$} &$(352, 0)$ &  $ \II_{(3, 0)}4^{-3}_{1}3^{-1} $  \\                                                                                 
\multirow{1}{*}{51a} & \multirow{1}{*}{${S}_4$}&$(168, 0)$ & $ \II_{(3, 2)}2^{-2}4^{3}_{3}3^{-1} $ &   \multirow{1}{*}{126b} &\multirow{1}{*}{$[384, 20089]$}&$(388, 0)$&  $ \II_{(3, 0)}4^{-3}_{1}3^{-1} $  \\                                                                                 
 \multirow{1}{*}{51b} & \multirow{1}{*}{${S}_4$}&$(190, 0)$  & $ \II_{(3, 2)}2^{2}4^{-3}_{7}3^{-1} $  & 127 & ${A}_6$&$(430, 0)$  &  $ \II_{(3, 0)}2^{2}4^{1}_{7}3^{2} $  \\                                                                              
52 & $F_{21}$&$(0, 0)$ & $ \II_{(3, 2)}4^{1}_{7}7^{-3} $ &  \multirow{1}{*}{128a} &\multirow{1}{*}{$\Gamma\textnormal{L}_2(\mathbb{F}_4)$} &$(0, 0)$&  $ \II_{(3, 0)}4^{1}_{7}3^{-2}5^{-2} $  \\                                                                                  
53 & $F_5$&$(0, 0)$ & $ \II_{(3, 2)}2^{-2}_{4}4^{-1}_{5}5^{3} $ &   \multirow{1}{*}{128b} &\multirow{1}{*}{$\Gamma\textnormal{L}_2(\mathbb{F}_4)$}&$(0, 0)$  &  $ \II_{(3, 0)}4^{1}_{7}3^{-2}5^{-2} $  \\                                                                        
54 & $C_2\times D_4$&$(116, 0)$ & $ \II_{(3, 2)}4^{-4}_{2}8^{-1}_{3} $ &  \multirow{1}{*}{129a} &\multirow{1}{*}{$C_2\times L_2(7)$} &$(0, 0)$ &  $ \II_{(3, 0)}2^{2}4^{1}_{7}7^{2} $  \\                                                                               
\multirow{1}{*}{55a} &\multirow{1}{*}{$QD_{16}$}&$(0, 0)$ & $ \II_{(3, 2)}2^{1}_{1}4^{-2}_{4}8^{-2} $  &  \multirow{1}{*}{129b} &\multirow{1}{*}{$C_2\times L_2(7)$} &$(0, 0)$ &  $ \II_{(3, 0)}2^{2}4^{1}_{7}7^{2} $  \\                                                                          
 \multirow{1}{*}{55b} &\multirow{1}{*}{$QD_{16}$}&$(222, 0)$ & $ \II_{(3, 2)}2^{3}_{1}8^{2} $  & 131 & $[192, 1494]$ &$(312, 0)$ &  $ \II_{(3, 0)}4^{-1}_{3}8^{-2}_{4} $  \\                                                                                    
 \multirow{1}{*}{55c} &\multirow{1}{*}{$QD_{16}$}&$(222, 40)$ & $ \II_{(3, 2)}2^{1}_{1}8^{2} $  & 132 & $T_{192}$&$(312, 48)$ &  $ \II_{(3, 0)}4^{3}_{3} $  \\                                                                                               
56 & ${A}_4$&$(144, 0)$& $ \II_{(3, 2)}4^{5}_{1} $ &  \multirow{1}{*}{133a} &\multirow{1}{*}{$C_2\times M_9$} &$(428, 0)$ &  $ \II_{(3, 0)}2^{2}_{2}4^{1}_{1}9^{1} $  \\                                                                                      
57 & $C_2^3$ &$(108, 16)$ & $ \II_{(3, 2)}4^{5}_{1} $ &   \multirow{1}{*}{133b} &\multirow{1}{*}{$C_2\times M_9$}&$(428, 0)$ &  $ \II_{(3, 0)}2^{-2}_{4}4^{-1}_{3}9^{1} $  \\                                                                                    
58 & $C_2^3$&$(146, 0)$ & $ \II_{(3, 2)}2^{2}_{0}4^{2}8^{1}_{1} $  & \multirow{1}{*}{134a} &\multirow{1}{*}{$\textnormal{A}\Gamma\textnormal{L}_1(\mathbb{F}_9)$} &$(0, 0)$ &  $ \II_{(3, 0)}2^{1}_{1}4^{-2}_{4}3^{1}9^{1} $  \\                                                                  
\multirow{1}{*}{59a} & \multirow{1}{*}{$C_2\times C_4$}&$(172, 0)$ & $ \II_{(3, 2)}2^{2}_{2}4^{-1}_{5}8^{-2}_{6} $ &   \multirow{1}{*}{134b} &\multirow{1}{*}{$\textnormal{A}\Gamma\textnormal{L}_1(\mathbb{F}_9)$} &$(0, 0)$&  $ \II_{(3, 0)}2^{1}_{1}4^{-2}_{4}3^{1}9^{1} $  \\                                                            
\multirow{1}{*}{59b} & \multirow{1}{*}{$C_2\times C_4$}&$(136, 0)$ & $ \II_{(3, 2)}2^{-2}_{6}4^{1}_{7}8^{-2}_{4} $  & \multirow{1}{*}{135a} &\multirow{1}{*}{${S}_3\times{S}_4$} &$(204, 0)$ &  $ \II_{(3, 0)}4^{3}_{3}3^{-2} $  \\                                                                          
60 & ${S}_3$&$(210, 0)$  & $ \II_{(3, 2)}2^{4}_{2}4^{1}_{7}3^{-2} $ &   \multirow{1}{*}{135b} &\multirow{1}{*}{${S}_3\times{S}_4$} &$(204, 0)$  &  $ \II_{(3, 0)}4^{3}_{3}3^{-2} $  \\                                                                               
63 & $C_6$&$(140, 0)$ & $ \II_{(3, 2)}2^{-4}_{2}4^{1}_{1}3^{3} $  & 136 & ${S}_5$&$(298, 0)$ &  $ \II_{(3, 0)}2^{2}4^{1}_{7}5^{-2} $  \\                                                                          
64 & $C_4$&$(84, 16)$ & $ \II_{(3, 2)}4^{5}_{1} $  & \multirow{1}{*}{137a} &\multirow{1}{*}{${S}_5$} &$(350, 0)$ &  $ \II_{(3, 0)}2^{2}_{0}4^{1}_{1}3^{-1}5^{-1} $  \\                                                                               
65 & $\#6144$ &$(448, 0)$ & $ \II_{(3, 1)}2^{2}4^{-1}_{3}8^{-1}_{3} $ &   \multirow{1}{*}{137b} &\multirow{1}{*}{${S}_5$} &$(350, 0)$  &  $ \II_{(3, 0)}2^{2}_{0}4^{1}_{1}3^{-1}5^{-1} $  \\                                                               
66 & $\#3072$ &$(384, 0)$  & $ \II_{(3, 1)}2^{2}4^{-2}_{4}3^{1} $  & 139 & $[96,195]$ &$(252, 0)$&  $ \II_{(3, 0)}4^{3}_{7}3^{2} $  \\                                                                                    
67 & $\#2048$&$(352, 0)$& $ \II_{(3, 1)}4^{4}_{2} $ &  143 & $[64,257]$ &$(270, 16)$ &  $ \II_{(3, 0)}2^{-1}_{3}8^{-2} $  \\                                                                                             
\multirow{1}{*}{69a} & \multirow{1}{*}{$M_{20}$}&$(0, 0)$& $ \II_{(3, 1)}2^{2}4^{-1}_{3}8^{1}_{7}5^{-1} $  & 146 & $C_2\times {S}_4$&$(290, 0)$ &  $ \II_{(3, 0)}4^{2}_{0}8^{-1}_{5}3^{-1} $  \\                                                               
 \multirow{1}{*}{69b} & \multirow{1}{*}{$M_{20}$}&$(400, 0)$ & $ \II_{(3, 1)}4^{1}_{1}8^{-1}_{5}5^{-1} $  & 148 & $C_2\times{S}_4$&$(296, 24)$ &  $ \II_{(3, 0)}2^{-2}_{4}8^{1}_{1}3^{-1} $  \\                                                                    
\multirow{1}{*}{70a} &\multirow{1}{*}{$F_{384}$} &$(0, 0)$ & $ \II_{(3, 1)}4^{-2}_{4}8^{-2}_{2} $  & \multirow{1}{*}{149a} &\multirow{1}{*}{$C_2\times F_5$} &$(232, 0)$ &  $ \II_{(3, 0)}2^{-2}_{4}4^{-1}_{3}5^{2} $  \\                                                                         
 \multirow{1}{*}{70b} &\multirow{1}{*}{$F_{384}$} &$(334, 0)$ & $ \II_{(3, 1)}2^{2}_{2}8^{2}_{0} $ &  \multirow{1}{*}{149b} &\multirow{1}{*}{$C_2\times F_5$} &$(232, 0)$ &  $ \II_{(3, 0)}2^{-2}_{4}4^{-1}_{3}5^{2} $  \\                                                                           
 \multirow{1}{*}{70c} &\multirow{1}{*}{$F_{384}$} &$(328, 0)$ & $ \II_{(3, 1)}2^{2}8^{2}_{2} $  & 152 & $C_2\times QD_{16}$ &$(340, 40)$&  $ \II_{(3, 0)}2^{-2}_{4}16^{-1}_{3} $  \\            
 \multirow{1}{*}{70d} &\multirow{1}{*}{$F_{384}$} &$(328, 48)$ & $ \II_{(3, 1)}8^{2}_{2} $  & 160& $C_2^2\rtimes C_4$&$(174, 24)$ &  $ \II_{(3, 0)}4^{-1}_{3}8^{-2}_{4} $  \\                                                                                         
71a & $[384, 20164]$ & $(328, 0)$& $ \II_{(3, 1)}4^{4}_{2} $ & &  & &   \\                                      
\hline                

\end{longtable}}

\bibliographystyle{halpha-abbrv}
\bibliography{references}

\end{document}